\documentclass[12pt]{amsart}
\usepackage{amsmath}
\usepackage{amsthm}
\usepackage{amsfonts}
\usepackage{amssymb}
\usepackage{tikz-cd}
\usepackage{graphicx}
\usepackage{hyperref}
\usepackage{url}
\usepackage{float}
\usepackage{comment}
\usepackage{lineno}
\usepackage{comment,MnSymbol,tikz}
\usetikzlibrary{backgrounds,arrows}
\usepackage{xcolor}

\newcommand{\meet}{\land}
\newcommand{\join}{\lor}

\newcommand{\isf}{\mathbf{IS}_f}
\newcommand{\is}{\mathbf{IS}}
\newcommand{\nis}{\mathbf{NIS}}
\newcommand{\nisf}{\mathbf{NIS}_f}
\newcommand{\pf}{\mathbf{P}_f}
\newcommand{\pfk}{\mathbf{P}_f^K}
\newcommand{\spf}{\mathbf{SP}_f}

\newcommand{\haf}{\mathbf{HA}_f}
\newcommand{\colorsof}{c}
\newcommand{\upsets}{{\sf Up}}
\newcommand{\up}{{\uparrow}}

\newcommand{\dn}{{\downarrow}}

\newcommand{\coverof}{\smalltriangledown}
\newcommand{\kmor}{K\"ohler morphism}
\newcommand{\smor}{S-morphism}
\newcommand{\isbot}{\mathbf{IS}^\bot}

\newcommand{\nisbot}{\mathbf{NIS}^\bot}
\newcommand{\nisfbot}{\mathbf{NIS}_f^\bot}

\newcommand{\spfbot}{\mathbf{SP}_f^\bot}

\let\ge\geqslant
\let\le\leqslant

\let\coverof\smalltriangledown

\newtheorem{theorem}{Theorem}[section]
\newtheorem{lemma}[theorem]{Lemma}
\newtheorem{proposition}[theorem]{Proposition}

\theoremstyle{definition}
\newtheorem{definition}[theorem]{Definition}

\newtheorem{example}[theorem]{Example}

\newtheorem{remark}[theorem]{Remark}
\newtheorem{convention}[theorem]{Convention}

\title{Diego's Theorem for nuclear implicative semilattices}

\author{G.~Bezhanishvili}
\address{New Mexico State University}
\email{guram@nmsu.edu}

\author{N.~Bezhanishvili}
\address{University of Amsterdam}
\email{N.Bezhanishvili@uva.nl}

\author{L.~Carai}
\address{New Mexico State University}
\email{lcarai@nmsu.edu}

\author{D.~Gabelaia}
\address{Tbilisi State University}
\email{gabelaia@gmail.com}

\author{S.~Ghilardi}
\address{University of Milan}
\email{silvio.ghilardi@unimi.it}

\author{M.~Jibladze}
\address{Tbilisi State University}
\email{mamuka.jibladze@gmail.com}


\begin{document}

\begin{abstract}
We prove that the variety of nuclear implicative semilattices is locally finite, thus generalizing Diego's Theorem. The key ingredients of our proof include the coloring technique and construction of universal models from modal logic. For this we develop duality theory for finite nuclear implicative semilattices, generalizing K\"ohler duality. We prove that our main result remains true for bounded nuclear implicative semilattices, give an alternative proof of Diego's Theorem, and provide an explicit description of the free cyclic nuclear implicative semilattice.
\end{abstract}

\maketitle

\tableofcontents

\section{Introduction}

It is a celebrated result of Diego \cite{Die66} that the variety $\is$ of implicative semilattices is locally finite. We prove that, surprisingly enough, Diego's Theorem remains true for the variety $\nis$ of nuclear implicative semilattices. A \emph{nucleus} on an implicative semilattice $A$ is a unary function $j:A\to A$ satisfying
\begin{itemize}
\item[(1)] $a\le ja$,
\item[(2)] $jj a=j a$,
\item[(3)] $j(a\land b)=ja\land jb$.
\end{itemize}
A \emph{nuclear implicative semilattice} is a pair $\mathfrak A=(A,j)$ where $A$ is an implicative semilattice and $j$ is a nucleus on $A$. Nuclei play an important role in different branches of mathematics, logic, and computer science:
\begin{itemize}
\item In topos theory, nuclei on the subobject classifier of a topos are exactly the Lawvere-Tierney operators, and give rise to sheaf subtoposes,
generalizing sheaves with respect to a Grothendieck topology \cite{Law72,Tie76}.
\item In pointfree topology, nuclei characterize sublocales of locales \cite{FS79,Joh82}.
\item In logic, nuclei model the so-called \emph{lax modality} \cite{FM97} (see also \cite{Gol81}).
As such, nuclear implicative semilattices provide algebraic semantics for the
$\vee$-free fragment of the \emph{Lax Logic} of \cite{FM97}, an intuitionistic modal logic with interesting links to computer science since lax modality is used to reason about formal verification of hardware \cite{Men91}.
\item In \cite{BH19} nuclei were used as a unifying tool for different semantics of intuitionistic logic.
\end{itemize}

Diego's proof that $\is$ is locally finite is algebraic, and it is unclear how to generalize it to $\nis$. Instead we develop a different technique based on duality theory and the coloring technique which allows the construction of universal models. The coloring technique was originally developed in \cite{EG77} to characterize dually when the Esakia space of a Heyting algebra (or an $\sf S4$-algebra) is finitely generated. Since then it has been used extensively in modal logic for constructing universal models (see, e.g., \cite{CZ97,Nic06}). In our considerations we will rely on the general method of \cite{Ghi95}.

Esakia duality is a standard tool for the study of Heyting algebras (see, e.g., \cite{Esa19}). Duality theory for implicative semilattices is more complicated than Esakia duality. For finite implicative semilattices it was developed
by K\"ohler \cite{Koh81}. It was generalized to the infinite case in \cite{VM86,Cel03,BJ13}. Since the dual structures arising in the infinite case
are more complicated to work with, we mostly concentrate on K\"ohler duality for finite implicative semilattices and generalize it to the setting of
finite nuclear implicative semilattices. We prove that $\nis$ has the finite model property, which allows us to mostly work with finite implicative
semilattices and their dual structures.

We generalize the coloring technique to this setting, which allows us to construct universal models for nuclear implicative semilattices.
We prove that the construction of the
$n$-universal model
terminates for each $n$, thus yielding that $\nis$ is locally finite. This generalizes Diego's Theorem to $\nis$. It also provides a different proof of Diego's Theorem for $\is$. While this different proof is more complicated
than Diego's original proof, it is this proof that admits a generalization to $\nis$. Whether Diego's technique generalizes to $\nis$ remains an interesting open problem.

We conclude the paper by showing how our results remain true if we add the bottom element $0$ to the signature of nuclear implicative semilattices, and giving the dual description of the free cyclic nuclear implicative semilattice.

We briefly compare the contributions of this paper with related results in the literature. As we pointed out above, $\bf NIS$ provides algebraic semantics for the $\vee$-free fragment of the Lax Logic.
We note that local finiteness of $\bf NIS$ is in stark contrast with local finiteness of the $\vee$-free and $\to$-free fragments of other intuitionistic modal logics. For example, one of the best studied intuitionistic modal logics is Prior's $\sf MIPC$, which axiomatizes the one-variable fragment of first-order intuitionistic logic (the same way $\sf S5$ axiomatizes the one-variable fragment of classical first-order logic). Neither the $\vee$-free fragment nor the $\to$-free fragment of $\sf MIPC$ is locally finite. In fact, the algebras corresponding to the frames $\mathfrak F_1$ and $\mathfrak F_2$ shown in \cite[Fig.~3 and~4]{BG01} are not locally finite in the $(\wedge,\to,\forall)$- and $(\wedge,\vee,\forall)$-signature, respectively.

A classic corollary of Diego's Theorem is McKay's Theorem \cite{McK68} that every intermediate logic axiomatized by $\vee$-free formulas has the finite model property. Our generalization of Diego's Theorem yields the following generalization of McKay's Theorem: every extension of the Lax Logic axiomatized by $\vee$-free formulas has the finite model property.

The paper is organized as follows. In Section~\ref{sec:nuclear implicative semilattices} we give basic definitions and facts about nuclear implicative semilattices, and prove that the variety of nuclear implicative semilattices is generated by its finite algebras. In Section~\ref{sec: finite duality implicative semilattices} we provide an alternative approach to K\"ohler duality for finite implicative semilattices, which relies heavily on the use of nuclei. In particular, unlike K\"ohler's approach, we describe both contravariant functors explicitly. Because of this, we add proofs where necessary. In Section~\ref{sec: finite duality nuclear implicative semilattices} we extend K\"ohler duality to the setting of nuclear implicative semilattices. In Section~\ref{sec:subalgebras} we characterize dually subalgebras and nuclear subalgebras of finite nuclear implicative semilattices.
We also provide a decomposition of implicative semilattice homomorphisms in the finite case, and show that a similar decomposition does not hold for nuclear homomorphisms.
In Section~\ref{sec:coloring} we utilize the results of Sections~\ref{sec: finite duality nuclear implicative semilattices} and~\ref{sec:subalgebras} to develop the coloring technique, which allows us to build $n$-universal models for the variety of nuclear implicative semilattices. We prove that the $n$-universal model is finite for each $n$, from which we derive our main result that the variety of nuclear implicative semilattices is locally finite, thus generalizing Diego's Theorem. In Section~\ref{sec:bounded case} we prove that the variety of bounded nuclear implicative semilattices remains locally finite, thus generalizing the main result of Section~\ref{sec:coloring}. In Section~\ref{sec:alternative Diego's Theorem} we give an alternative proof of Diego's Theorem by using a strategy analogous to that in Section~\ref{sec:coloring}.
The most challenging part
is to prove that the variety of implicative semilattices is generated by its finite algebras without relying on Diego's Theorem. Finally, in Section~\ref{sec:examples}
we describe the $n$-universal models and free $n$-generated nuclear implicative semilattices for small $n$.

\section{Nuclear implicative semilattices}\label{sec:nuclear implicative semilattices}

We recall that a \textit{meet-semilattice} is an algebra $\mathfrak A=(A,\meet)$ where the binary operation $\meet:A^2\to A$ is associative, commutative, and idempotent. If $(A,\wedge)$ is a meet-semilatice, then we can define a partial order $\le$ on $A$ by $a \le b$ iff $a= a \meet b$. Then $a\meet b$ is the greatest lower bound of $\{a,b\}$, and meet-semilattices can be defined alternatively as partially ordered sets $(A,\le)$ such that each finite subset of $A$ has a greatest lower bound.

\begin{definition}
An \textit{implicative semilattice} is an algebra $\mathfrak A=(A, \meet, \to,1)$ where $(A,\wedge,1)$ is a meet-semilattice with a greatest element
and the binary operation $\to:A^2\to A$ satisfies
\[
a \le b \to c \quad \mbox{ iff } a \meet b \le c.
\]
\end{definition}

\begin{remark}\label{rem:bounds}
\begin{enumerate}
\item[]
\item In each implicative semilattice we have that $1=a \to a$, and that $a \leq b$ iff $a \to b=1$. On the other hand, an implicative semilattice may not have a least element.
\item It is well known that implicative semilattices can be defined equationally, and hence they form a variety (see, e.g., \cite{Koh81}).
\item Every Heyting algebra is clearly an implicative semilattice. The converse is not true in general. However, every finite implicative semilattice is a Heyting algebra.
\end{enumerate}
\end{remark}

\begin{definition}
A map between two implicative semilattices is an \textit{implicative semilattice homomorphism} provided it preserves the operations $\meet$ and $\to$. Let $\is$ denote the category of implicative semilattices and  homomorphisms between them.
\end{definition}

\begin{remark}
Although finite implicative semilattices are Heyting algebras, implicative semilattice homomorphisms do not have to preserve finite joins, and hence may not be Heyting algebra homomorphisms.
\end{remark}

The well-known correspondence between congruences and filters of Heyting algebras extends to implicative semilattices. Thus, as with Heyting algebras, an implicative semilattice is subdirectly irreducible iff it has the second largest element (see, e.g., \cite{Koh81}).

\begin{definition}
Let $(A,\meet,\to)$ be an implicative semilattice.
\begin{enumerate}
\item A subset $B$ of $A$ is a \textit{subalgebra} of $A$ if it closed under $\meet,\to,1$.
\item A subalgebra $B$ of $A$ is a \textit{total subalgebra} if $a \in A$ and $b \in B$ imply $a \to b \in B$.
\end{enumerate}
\end{definition}

We next recall from the introduction that a nucleus $j$ on an implicative semilattice (or more generally on a meet-semilattice) $A$ is a unary function $j:A\to A$ that is inflationary ($a \le j(a)$), idempotent ($j(a)=j(j(a))$) and meet-preserving ($j(a \meet b)=j(a) \meet j(b)$).

As we pointed out in the introduction, nuclei play a fundamental role in pointfree topology as they characterize sublocales or equivalently regular epimorphisms of frames. It is well known that all nuclei on a frame also form a frame. The following are well-known nuclei on a frame:
\begin{itemize}
\item $c_a(b)=a \join b$ (a \textit{closed} nucleus);
\item $o_a(b)=a \to b$ (an \textit{open} nucleus);
\item $w_a(b)=(b \to a ) \to a$.
\end{itemize}
It is well known that each nucleus on a frame is a join of $c_a\wedge o_b$, and a meet of $w_a$ (see, e.g., \cite{Joh82,picado2012frames}). Nuclei also play an important role in the semantic hierarchy of intuitionistic logic \cite{BH19}.

\begin{definition}
For an implicative semilattice $A$ and a nucleus $j$ on it, let
\[
A_j = \{ j(a) \mid a\in A\}.
\]
\end{definition}

It is easy to see that $A_j$ is the set of fixpoints of $j$; that is,
\[
A_j=\{ a \in A \mid j(a)=a \}.
\]

The next proposition, which is well known, relates total subalgebras and fixpoints of nuclei on implicative semilattices.

\begin{proposition} \label{prop:nucleus induced by subalgebra}
Let $A$ be an implicative semilattice.
\begin{enumerate}
\item If $j$ is a nucleus on $A$, then $A_j$ is a total subalgebra of $A$.
\item For a finite subalgebra $B$ of $A$, define $k$ on $A$ by $k = \bigwedge\{w_b \mid b \in B\}$.
\begin{enumerate}
\item $k$ is a nucleus on $A$.
\item $B$ is a Heyting subalgebra of $(A_k, \meet, \join_k, \to, 0_k)$ where $a\join_k b=k(a\vee b)$ and $0_k=k(0)$.
\item If $B$ is a total subalgebra of $A$, then $B=A_k$.
\end{enumerate}
\end{enumerate}
\end{proposition}

\begin{proof}
For (1) see, e.g., \cite[Rem.~8]{BG07}; for (2a) and (2b) see, e.g., \cite[Prop.~36]{BG07}. To see (2c), it remains to show that $A_k\subseteq B$. Let $a \in A_k$. Then $a=k(a)=\bigwedge\{(a \to b) \to b \mid b \in B\}$. Since $B$ is a total subalgebra of $A$, we have that $a \to b \in B$ for all $b \in B$. Thus, $k(a)\in B$.
\end{proof}

The next definition is central to the paper.

\begin{definition}
A \textit{nuclear implicative semilattice} is an algebra $\mathfrak A=(A,\meet,\to,j)$ where $(A,\meet,\to)$ is an implicative semilattice and $j$ is a nucleus on $A$.
\end{definition}

Clearly nuclear implicative semilattices form a variety.

\begin{definition}
An implicative semilattice homomorphism between two nuclear implicative semilattices is called a
\textit{nuclear homomorphism}
provided it preserves $j$. Let $\nis$
be
the category of nuclear implicative semilattices and
nuclear
homomorphisms between them.
\end{definition}

Since each nucleus $j$ is inflationary, filters are always closed under $j$. Thus, we obtain the following characterization of congruences and subdirectly irreducible nuclear implicative semilattices.

\begin{proposition}\label{prop:hom images-filters in nis}
Congruences of a nuclear implicative semilattice $\mathfrak A$ correspond to filters of $\mathfrak A$. Therefore, $\mathfrak A$ is subdirectly irreducible iff it is subdirectly irreducible as an implicative semilattice $($which happens iff $\mathfrak A$ has the second largest element$)$.
\end{proposition}

Nevertheless, Diego's proof
\cite{Die66} (see also \cite{Koh81})
does not generalize directly to the setting of nuclear implicative semilattices. The key difference is that if $\mathfrak A$ is a subdirectly irreducible nuclear implicative semilattice, then the subset obtained by removing the second largest element of $\mathfrak A$, while closed under $\meet$ and $\to$, may not be closed under $j$.

\begin{definition}
Let $\mathfrak{A}=(A,\meet,\to,j)$ be a nuclear implicative semilattice.
\begin{enumerate}
\item A subalgebra of $(A,\meet,\to)$ is called a \textit{nuclear subalgebra} if it is closed under $j$.
\item If $\mathfrak B$ is a total subalgebra and a nuclear subalgebra of $\mathfrak A$, then we call it a \emph{total nuclear subalgebra} of $\mathfrak A$.
\end{enumerate}
\end{definition}

As the first step towards proving that $\nis$ is locally finite, we show that $\nis$ is generated by its finite algebras. For this we utilize Diego's Theorem that $\is$ is locally finite. To see that $\nis$ is generated by its finite algebras, it is sufficient to show that each equation $t(x_1, \ldots, x_n) =1$ that is not derivable from the equations defining $\nis$ is refuted in some finite nuclear implicative semilattice.

\begin{theorem}\label{thm:nis generated by finite algebras}
The variety $\nis$ is generated by its finite algebras.
\end{theorem}

\begin{proof}
Let $t(x_1, \ldots,x_n)$ be a term in the language of nuclear implicative semilattices such that the equation $t(x_1, \ldots, x_n) =1$ is not derivable from the equations defining $\nis$. Then there is a nuclear implicative semilattice $\mathfrak A=(A,\meet,\to,j)$ and $a_1, \ldots,a_n \in A$ such that $t(a_1, \ldots, a_n) \neq 1$ in $\mathfrak A$. Set
\[
F=\{ t'(a_1, \ldots,a_n) \mid \mbox{$t'$ is a subterm of $t$} \}.
\]
Then $F$ is a finite subset of $A$. Let $B$ be the subalgebra of $(A,\meet,\to)$ generated by $F$. By Diego's Theorem, $B$ is finite. Define $j_B$ on $B$ by
\[
j_B(b)= \bigwedge \{ x \in B \cap A_j \mid b \le x \}.
\]
We clearly have that $j(b)\le j_B(b)$, and that if $j(b)\in B$, then $j(b)=j_B(b)$. We show that $j_B$ is a nucleus on $B$. By definition, $b \le j_B(b)$. Also $j_B(j_B(b))=j_B(b)$ because $j_B(b) \in B \cap A_j$. For $a,b\in B$ we have
\begin{align*}
j_B(a) \meet j_B(b) & = \bigwedge \{ x \in B \cap A_j \mid a \le x \} \meet \bigwedge \{ y \in B \cap A_j \mid b \le y \}  \\
& = \bigwedge \{ x \meet y \mid x,y \in B \cap A_j, \: a \le x, \, b \le y \}.
\end{align*}
On the other hand,
\begin{align*}
j_B(a \meet b) & = \bigwedge \{ z \in B \cap A_j \mid a \meet b \le z \}.
\end{align*}
We show that
\[
\{ x \meet y \mid x,y \in B \cap A_j, \: a \le x, \, b \le y \}=\{ z \in B \cap A_j \mid a \meet b \le z \}.
\]
The left-to-right inclusion is clear. The right-to-left inclusion is a consequence of the fact that every implicative semilattice is a \textit{distributive semilattice}; that is, if $a \meet b \le z$, then there are $x \ge a$ and $y \ge b$
with
$z=x \meet y$. As follows from \cite[Prop.~2.1]{BJ08}, the elements $x,y$ can be taken to be
\[
x=(((a \to z) \meet (b \to z))\to z)\meet (b \to z) \quad \mbox{ and } \quad y=(((a \to z) \meet (b \to z))\to z)\meet (a \to z).
\]
Observe that $x,y \in B$ because $a,b,z\in B$ and $B$ is a subalgebra of $A$, and $x,y \in A_j$ since $z\in A_j$ and $A_j$ is a total subalgebra of $A$. Therefore, $j_B(a \meet b) = j_B(a) \meet j_B(b)$, and hence $j_B$ is a nucleus on $B$. Thus, $(B,j_B)$ is a finite nuclear implicative semilattice (although it may not be a nuclear subalgebra of $\mathfrak A$).

Since $B$ is a subalgebra of $(A,\meet,\to)$ and $j(b) \in B$ implies $j_B(b)=j(b)$, for each subterm $t'$ of $t$, the computation of $t'(a_1, \ldots, a_n)$ in $\mathfrak A$ is the same as that in $(B,j_B)$. Therefore, $t(a_1, \ldots, a_n) \neq 1$ in $\mathfrak A$ implies that $t(a_1, \ldots, a_n) \neq 1$ in $(B,j_B)$. Thus, $t(x_1, \ldots, x_n) =1$ is refuted in a finite nuclear implicative semilattice.
\end{proof}

\section{K\"ohler duality for finite implicative semilattices} \label{sec: finite duality implicative semilattices}

In this section we recall K\"ohler duality \cite{Koh81} for finite implicative semilattices. Our approach is different from K\"ohler's in that we explicitly define the functor from the category of finite implicative semilattices. We also work with nuclei instead of total subalgebras, and follow the standard approach in logic in working with upsets instead of downsets of a poset. Because of these differences, we provide details where necessary.

We start by recalling the well-known duality for finite Heyting algebras. It is a consequence of Esakia duality \cite{Esa74} for all Heyting algebras, but can also be derived directly (see, e.g., \cite{JT66}).

\begin{definition}
Let $(X, \leq)$ be a poset (partially ordered set). For $U\subseteq X$ let
\begin{eqnarray*}
\up U &=& \{ x \in X \mid \exists u \in U : u \le x \} \\
\dn U &=& \{ x \in X \mid \exists u \in U : x \le u \}.
\end{eqnarray*}
If $U=\{x\}$, we simply write $\up x$ and $\dn x$. We call $U$ an \textit{upset} if $U=\up U$ and a \emph{downset} if $U = \dn U$.

The set $\upsets(X)$ of all upsets of $X$ ordered by inclusion has naturally the structure of a Heyting algebra in which the meet and join are set-theoretic intersection and union, and
\[
U \to V = X \setminus \dn (U \setminus V) = \{ x \in X \mid (\forall y \ge x) (y \in U \Rightarrow y \in V) \}.
\]
\end{definition}

\begin{definition}
A map $f$ between two posets $(X,\le)$ and $(Y,\le)$ is a \emph{p-morphism}
(or \textit{bounded morphism})
if
\begin{itemize}
\item $x \le x'$ implies $f(x) \le f(x')$;
\item $f(x) \le y$ implies that there is $z \in X$ with $x\le z$ and $f(z)=y$.
\end{itemize}
\end{definition}

\begin{definition}
Let $\haf$ be the category of finite Heyting algebras and Heyting algebra homomorphisms, and let $\pf$ be the category of finite posets and p-morphisms.
\end{definition}

We then have the following well-known theorem.

\begin{theorem}\label{thm:haf}
$\haf$ is dually equivalent to $\pf$.
\end{theorem}

This duality is obtained by the contravariant functors $(\;)_*:\haf \to \pf$ and $(\;)^*:\pf \to \haf$. The functor $(\;)^*$ associates with each finite poset  $(X, \le)$ the Heyting algebra $X^*=\upsets(X)$; and with each p-morphism $f:X \to Y$ the Heyting algebra homomorphism $f^*:\upsets(Y) \to \upsets(X)$ given by $f^*(V)=f^{-1}(V)$.

The functor $(\;)_*$ is usually defined by associating with each finite Heyting algebra $A$ the poset of join-prime elements of $A$. Since we will mainly work in the signature of meet-semilattices, we will instead work with meet-prime elements.

\begin{definition}
Let $A$ be a meet-semilattice. An element $m\in A\setminus\{1\}$ is \emph{meet-prime} if $a \meet b \le m$ implies that $a \le m$ or $b \le m$. Let $X_A$ be the set of meet-prime elements of $A$. We let $\sqsubseteq$ be the dual of the restriction of $\le$ to $X_A$; that is, $m \sqsubseteq n$ iff $n \leq m$ in $A$. Then $(X_A,\sqsubseteq)$ is a poset.
\end{definition}

The functor $(\;)_*$ is defined by associating with each finite Heyting algebra $A$ the poset $A_*=(X_A,\sqsubseteq)$; and with each Heyting algebra homomorphism $h:A \to B$ the p-morphism $h_*:X_B \to X_A$ given by
\[
h_*(y)=\bigvee \{ a \in A \mid h(a) \le y \};
\]
that is, $h_*$ is the right adjoint of $h$ restricted to the set of meet-prime elements.

The functors $(\;)_*,(\;)^*$ yield a dual adjunction, where the natural isomorphisms
\[
\alpha_A:A \to \upsets(X_A) \mbox{ and } \varepsilon_X:X \to X_{\upsets(X)}
\]
are given by
\begin{equation}\label{eq:epsilon}\tag{$\dagger$}
\alpha_A(a)=\{ x \in X_A \mid a \nleq x \} \mbox{ and } \varepsilon_X(x)=X \setminus \dn x.
\end{equation}
This gives the desired dual equivalence between $\haf$ and $\pf$.

We next extend this duality to the setting of finite implicative semilattices. Since each finite implicative semilattice is a Heyting algebra, it is isomorphic to the Heyting algebra of upsets of a finite poset. Thus, at the object level, the duality for finite implicative semilattices is the same as for finite Heyting algebras. The key difference is in describing dually implicative semilattice homomorphisms by means of special partial p-morphisms.

\begin{definition}
Let $\isf$ be the full subcategory of $\is$ consisting of finite implicative semilattices.
\end{definition}

Let $(X,\le)$ be a poset. As usual, for $x,y\in X$ we write $x<y$ if $x\le y$ and $x\ne y$. For a partial function $f$ between posets, we denote by $D$ the domain of $f$.

\begin{definition}\label{def:kmor}
Let $(X, \le)$ and $(Y, \le)$ be two posets. We call a partial function $f:X \to Y$ a \textit{\kmor{}} if for each $x, x' \in D$ and $y \in Y$ we have
\begin{enumerate}
\item $x < x'$ implies $f(x) < f(x')$;
\item $f(x) < y$ implies that there is $z \in D$ with $x<z$ and $f(z)=y$.
\end{enumerate}
\end{definition}

\begin{remark}\label{rem:p-morphism}
\begin{enumerate}
\item[]
\item In Definition~\ref{def:kmor}(2), replacing $<$ with $\le$ results in an equivalent condition. However, Definition~\ref{def:kmor}(1) is strictly stronger than the corresponding condition involving $\le$.
\item If $f:X \to Y$ is a \kmor{}, then $f(\up x)$ is an upset of $Y$ for each $x \in X$, and so $U$ an upset of $X$ implies that $f(U)$ is an upset of $Y$. If in addition $x \in D$, then $f(\up x)=\up f(x)$.
\end{enumerate}
\end{remark}

\begin{definition}
Let $\pfk$ be the category of finite posets and \kmor{}s. If $f: X \to Y$ and $g:Y \to Z$ are \kmor{}s with domains $D$ and $D'$, then set-theoretic composition
$gf:X \to Z$
is a \kmor{} with the domain $f^{-1}(D') \subseteq D$. Thus, identity morphisms are total identity functions.
\end{definition}

\begin{definition}
Let $(X,\le)$ be a poset and $A\subseteq X$. As usual, $x \in A$ is called a \textit{maximal} point of $A$ if $x \le a$ implies $x=a$ for each $a \in A$. Minimal points are defined dually. Let $\max A$ be the set of maximal points and $\min A$ the set of minimal points of $A$.
\end{definition}

\begin{remark}
Since every finite implicative semilattice is a Heyting algebra, $\haf$ is a wide subcategory of $\isf$ (meaning that $\haf$ and $\isf$ have the same objects). On the other hand, $\pf$ is not a subcategory of $\pfk$ as not every p-morphism satisfies Definition~\ref{def:kmor}(1). Nevertheless, $\pf$ is isomorphic to a wide subcategory of $\pfk$, which can be seen as follows. With each p-morphism $f:X \to Y$ we can associate the \kmor{} by restricting the domain of $f$ to the set $D=\{x \in X \mid x \in \max f^{-1}(f(x))\}$. This induces an isomorphism between $\pf$ and the wide subcategory of $\pfk$ given by the \kmor{}s $f:X \to Y$ that satisfy $f(\up x)$ is a principal upset for each $x \in X$.
An alternate duality to K\"ohler duality is developed in \cite{BB09} (see also \cite[Sec.~5]{BJ13}), where $\pf$ is indeed a wide subcategory of the dual category to $\isf$.
\end{remark}

The duality for finite Heyting algebras then extends to the following duality for finite implicative semilattices.

\begin{theorem} [K\"ohler duality]
$\isf$ is dually equivalent to $\pfk$.
\end{theorem}

The object level of K\"ohler duality follows from Theorem~\ref{thm:haf}. To extend it to morphisms, it is convenient to first recall the dual characterization of nuclei from \cite{BG07}. While \cite{BG07} gives a dual characterization of nuclei on arbitrary Heyting algebras, we will restrict ourselves to the finite case.

For a poset $X$ and $S\subseteq X$ define $j_S$ on $\upsets(X)$ by
\[
j_S(U)=X \setminus \dn (S \setminus U).
\]
It is easy to check (see also \cite{BG07}) that $j_S$ is a nucleus on $\upsets(X)$. Conversely, suppose $A$ is a finite implicative semilattice and $j$ is a nucleus on $A$. Define a subset $S_j$ of $X_A$ by
\[
S_j=X_A \cap A_j.
\]
Thus, $S_j$ is the set of meet-primes of $A$ that are fixpoints of $j$. The next lemma, as well as Lemma~\ref{lem:alpha_A}, follow from \cite{BG07}, but it is easy to give their direct proofs.

\begin{lemma}\label{lem:S_j}
$X_{A_j}=S_j$.
\end{lemma}

\begin{proof}
Let $x \in A_j$. If $x$ is a meet-prime element of $A$, then it is clearly a meet-prime element of $A_j$. Conversely, suppose that $x$ is a meet-prime element of $A_j$ and $a \meet b \le x$ for some $a,b \in A$. Then $j(a) \meet j(b)=j(a \meet b) \le j(x)=x$. Since $j(a),j(b) \in A_j$, we have $j(a) \le x$ or $j(b) \le x$. Therefore, $a \le x$ or $b \le x$, and so $x$ is a meet-prime element of $A$. Thus, $X_{A_j}=S_j$.
\end{proof}

For $a \in A$ we call the minimal elements of $\{ x \in X_A \mid a \le x \}$ the \textit{meet-prime components} of $a$.
Since $A$ is finite, $a= \bigwedge \{ x \in X_A \mid a \le x \}$, so $a$ is the meet of its meet-prime components.

\begin{lemma}\label{lem:components of fixpoints}
Let $A$ be a finite implicative semilattice and $j$ a nucleus on $A$. If $a \in A_j$, then the meet-prime components of $a$ are also in $A_j$.
\end{lemma}

\begin{proof}
We have $a=x_1 \meet \cdots \meet x_n$ where $x_1,\ldots,x_n$ are the meet-prime components of $a$. Therefore,
\[
j(x_1) \meet \cdots \meet j(x_n)=j(x_1 \meet \cdots \meet x_n)=j(a)=a \le x_1.
\]
Since $x_1$ is meet-prime, $j(x_i) \le x_1$ for some $i$. Thus, $x_i \le j(x_i) \le x_1$. By minimality of $x_1$, we have that $x_i=j(x_i)=x_1$. In particular, $x_1 \in A_j$. A similar argument yields that $x_i \in A_j$ for each $i$.
\end{proof}

\begin{lemma} \label{lem:alpha_A}
$\alpha_A(ja)=j_{S_j}\alpha_A(a)$.
\end{lemma}

\begin{proof}
We recall that the isomorphism $\alpha_A:A \to \upsets(X_A)$ is given by $\alpha_A(a)=\{ x \in X_A \mid a \nleq x \}$ and that the order $\sqsubseteq$ on $X_A$ is the dual of $\le$. We have $x\notin\alpha_A(ja)$ iff $j(a)\le x$ and $x\notin j_{S_j}\alpha_A(a)$ iff there is $y\in S_j$ with $a\le y$ and $y \le x$. Since $y\in S_j$ implies that $j(y)=y$, the existence of such $y$ implies that $j(a)\le x$. Conversely, suppose that $j(a)\le x$. Then there is a meet-prime component
$y$ of $j(a)$ such that $y\le x$
By Lemma~\ref{lem:components of fixpoints}, $j(y)=y$. Thus, $y \in S_j$ and $a \le y \le x$.
\end{proof}

As a result we obtain the following representation of finite nuclear implicative semilattices (which is also a consequence of \cite{BG07}).

\begin{theorem} \label{thm: (A,j) iso to (A,j)_*^*}
Let $(A,j)$ be a finite nuclear implicative semilattice. Then $(A,j)$ is isomorphic to $(\upsets(X_A),j_{S_j})$.
\end{theorem}

We are ready to define contravariant functors $(\;)^*: \pfk \to \isf$ and $(\;)_*: \isf \to \pfk$ which yield K\"ohler duality.

We define $(\;)^*: \pfk \to \isf$ on objects by sending each finite poset $X$ to $X^*=\upsets(X)$. If $f: X \to Y$ is a \kmor{}, let $f^*: \upsets(Y) \to \upsets(X)$ be given by
\[
f^*(V)=X \setminus \dn f^{-1}(Y \setminus V).
\]
Using the definition of \kmor{}s, it is straightforward to show that $f^*$ is an implicative semilattice homomorphism and that $(\;)^*$ reverses the order of compositions. It is also clear that $(\;)^*$ preserves identity \kmor{}s. Thus, $(\;)^*: \pfk \to \isf$ is a well-defined contravariant functor.

We define $(\;)_*: \isf \to \pfk $ on objects by sending each finite implicative semilattice $A$ to $A_*=(X_A,\sqsubseteq)$. Let $h: A \to B$ be an implicative semilattice homomorphism. Then $h(A)$ is a subalgebra of $B$, so by Proposition~\ref{prop:nucleus induced by subalgebra}(2a), $h(A)$ gives rise to the nucleus $j=\bigwedge\{w_{h(a)}\mid a\in A\}$ on $B$. Let $S_j=X_B \cap B_j$.

\begin{remark}\label{rem:S_j}
We have
\begin{align*}
S_j &=\{ y \in X_B \mid y= \bigwedge_{a \in A} w_{h(a)} (y) \} =\{ y \in X_B \mid y= \bigwedge_{a \in A} (y \to h(a)) \to h(a) \} \\
&=\{ y \in X_B \mid y= (y \to h(a)) \to h(a) \mbox{ for some } a \in A \}\\
&=\{ y \in X_B \mid y= b \to h(a)  \mbox{ for some } b \in B, \, a \in A  \}
\end{align*}
\end{remark}

\begin{lemma}
$h:A\to B_j$ is a Heyting algebra homomorphism.
\end{lemma}

\begin{proof}
The map $h:A \to B_j$ is the composition of the onto homomorphism $h:A \to h(A)$ and the inclusion $h(A) \hookrightarrow B_j$. Since $h:A \to h(A)$ is onto, it is determined by a filter, hence is a Heyting algebra homomorphism. That $h(A) \hookrightarrow B_j$ is a Heyting algebra homomorphism follows from Proposition~\ref{prop:nucleus induced by subalgebra}(2b). Thus, $h:A \to B_j$ is a Heyting algebra homomorphism.
\end{proof}

Since $h:A \to B_j$ is a Heyting algebra homomorphism, it has a right adjoint $h_*:B_j \to A$ given by
\[
h_*(y)=\bigvee \{ a \in A \mid h(a) \le y \}.
\]
Therefore, $h_*$ maps $X_{B_j}$ to $X_A$. By Lemma~\ref{lem:S_j}, $X_{B_j}=S_j$. Thus, $h_*$ restricts to a map $h_*:S_j \to X_A$.
As a result, we obtain a partial map $h_*:X_B \to X_A$ with domain $S_j$.

\begin{lemma}\label{lem:kohler morph}
The partial map $h_*:X_B \to X_A$ is a \kmor{}.
\end{lemma}

\begin{proof}
It follows from Theorem~\ref{thm:haf} that $h_*:S_j \to X_A$ is a p-morphism.
Therefore, by Remark~\ref{rem:p-morphism}(1), it is sufficient to show that if $x,y \in S_j$ with $x < y$, then $h_*(x) < h_*(y)$. Since $x<y$ implies $x\le y$ and $h_*:S_j\to X_A$ is a p-morphism, we have $h_*(x) \le h_*(y)$. Suppose that $h_*(x) = h_*(y)$. Since $y \in S_j$,
by Remark~\ref{rem:S_j},
$y=b \to h(a)$ for some $b \in B$ and $a\in A$. Thus, $h(a) \le b \to h(a)=y$.
Because $h_*$ is right adjoint to $h$,
this implies $a \le h_*(y)=h_*(x)$, so $h(a) \le x$. From $y=b \to h(a)$ it follows that $y \meet b \le h(a) \le x$. Since $x$ is meet-prime and $x < y$, we have $b \le x$,
so $b \le y$.
Therefore,
\[
1= b \to y=b \to (b \to h(a))=b \to h(a)=y,
\]
which is a contradiction since $y$ is meet-prime. Thus, $h_*(x) < h_*(y)$.
\end{proof}

Consequently, we can define $(\;)_*$ on morphisms by sending $h$ to $h_*$.

\begin{lemma}
$(\;)_*$ is a contravariant functor.
\end{lemma}

\begin{proof}
It is easy to see that
$(\;)_*$ preserves identity homomorphisms. Let $h:A \to B$ and $g:B \to C$ be homomorphisms between finite implicative semilattices. To see that $(gh)_*=h_* g_*$ let $j=\bigwedge\{w_{h(a)}\mid a\in A\}$ be the nucleus on $B$ corresponding to the subalgebra $h(A)$, and let $k=\bigwedge\{w_{g(b)}\mid b\in B\}$ and $l=\bigwedge\{w_{g(h(a))}\mid a\in A\}$ be the nuclei on $C$  corresponding to the subalgebras $g(B)$ and $g(h(A))$. The domain of $h_*$ is then $S_j\subseteq X_B$ and the domains of $g_*,(gh)_*$ are $S_k, S_l\subseteq X_C$. Since
$gh(A) \subseteq g(B)$,
we have $k \le l$, which implies that $A_l \subseteq A_k$ and $S_l \subseteq S_k$.

We show that $S_l=g_*^{-1}(S_j)$, yielding that the domain of $(gh)_*$ coincides with the domain of $h_* g_*$. Let $x \in S_l$.
By Remark~\ref{rem:S_j},
$x=(x \to gh(a)) \to gh(a)$ for some $a \in A$. Since $S_l \subseteq S_k$, we have $x \in S_k \subseteq C_k$. Therefore, $gg_* (x) \le x$ because $g_*$
is right adjoint to
$g:B \to C_k$. Thus,
\begin{align*}
g((g_*(x) \to h(a)) \to h(a)) & =(gg_*(x)\to gh(a)) \to gh(a)\\
& \le (x\to gh(a)) \to gh(a)=x.
\end{align*}
The above inequality implies $(g_*(x) \to h(a)) \to h(a) \le g_*(x)$ which gives $g_*(x)=(g_*(x) \to h(a)) \to h(a)$. Therefore, $g_*(x) \in S_j$, and so  $S_l \subseteq g_*^{-1}(S_j)$.

To show the other inclusion, let $x \in g_*^{-1}(S_j)$. Then $x \in S_k$ and so $x=c \to g(b)$ for some $c \in C$ and $b \in B$. Also, $g_*(x)=b' \to h(a)$ for some $b' \in B$ and $a \in A$ because $g_*(x) \in S_j$. Since $x \in C_k$, we have $gg_*(x) \le x$. So
\begin{align*}
c \to gg_*(x) \le c \to x=c \to (c \to g(b))=c \to g(b)=x.
\end{align*}
Since $g(b) \le c \to g(b)=x$ and $x \in C_k$, we have $b \le g_*(x)$ and so $g(b) \le gg_*(x)$. Therefore,
\begin{align*}
x=c \to g(b) \le c \to gg_*(x).
\end{align*}
Thus, $x=c \to gg_*(x)$ which gives
\begin{align*}
x &=c \to gg_*(x)=c \to g(b' \to h(a))\\
&=c \to (g(b') \to gh(a))=(c \meet g(b')) \to gh(a).
\end{align*}
Since $S_l$ is the set of the meet-primes of the form
$c' \to gh(a)$ for some $c' \in C$
and $a \in A$, we have that $x \in S_l$. Consequently, $S_l=g_*^{-1}(S_j)$.

It remains to show that if $x \in S_l=g_*^{-1}(S_j)$, then $(gh)_*(x)=h_*g_*(x)$. Let $a \in A$. Since $x \in S_l$, we have
\[
a \le (gh)_*(x) \mbox{ iff }  gh(a) \le x.
\]
From $x \in S_l \subseteq S_k$ and $g_*(x) \in S_j$ it follows that
\[
a \le h_*g_*(x) \mbox{ iff } h(a) \le g_*(x)  \mbox{ iff }  gh(a) \le x.
\]
This implies that $a \le (gh)_*(x)$ iff $a \le h_*g_*(x)$ for each $a \in A$. Thus, $(gh)_*(x)=h_*g_*(x)$.
\end{proof}

Finally, since $\alpha_A$ and $\varepsilon_X$ are natural isomorphisms, the functors $(\;)_*: \isf \to \pfk$ and $(\;)^*: \pfk \to \isf$ yield a dual equivalence of $\isf$ and $\pfk$, concluding the proof of K\"ohler duality.

\section{Duality for finite nuclear implicative semilattices}\label{sec: finite duality nuclear implicative semilattices}

In this section we generalize K\"ohler duality to the setting of nuclear implicative semilattices. Let $(X,\le)$ be a finite poset. As we pointed out in Section~\ref{sec: finite duality implicative semilattices}, each subset $S$ of $X$ gives rise to a nucleus on $\upsets(X)$ given by $j_S(U)= X \setminus \dn (S \setminus U)$. Conversely, to each nucleus $j$ on a finite implicative semilattice $A$ there corresponds a subset of $X_A$ given by $S_j=X_A \cap A_j$. By Theorem~\ref{thm: (A,j) iso to (A,j)_*^*}, $\alpha_A$ is a nuclear implicative semilattice isomorphism between $(A,j)$ and $(\upsets(X_A), j_{S_j})$. We will extend this representation result to a full duality.

\begin{definition}
Let $\nisf$ be the full subcategory of $\nis$ consisting of finite nuclear implicative semilattices.
\end{definition}

\begin{definition}
We call a pair $(X,S)$ an \textit{S-poset} if $(X,\le)$ is a poset and $S$ is a subset of $X$.
\end{definition}

\begin{lemma} \label{prop:S-morph}
Let $(X,S)$ and $(Y,T)$ be two finite S-posets, $f:X \to Y$ a \kmor{} with domain $D \subseteq X$, and $f^*:\upsets(Y) \to \upsets(X)$ its dual implicative semilattice homomorphism. Then $f^*$ is a nuclear homomorphism iff for all $x \in X$ we have
\begin{equation*}\tag{$\ast$}
\up (f(\up x) \cap T)=f(\up(\up x \cap S)).
\end{equation*}
\end{lemma}

\begin{proof}
For $x \in X$ and $V \in \upsets(Y)$, we have
\begin{align*}
x \in f^*(j_T(V)) & \mbox{ iff } x \in X \setminus \dn f^{-1} (Y \setminus j_T(V)) \\
& \mbox{ iff } \up x \cap f^{-1}(Y \setminus j_T(V))= \emptyset \\
& \mbox{ iff } f(\up x) \subseteq j_T(V)\\
& \mbox{ iff } f(\up x) \subseteq Y \setminus \dn (T \setminus V)\\
& \mbox{ iff } f(\up x) \cap \dn(T \setminus V)= \emptyset\\
& \mbox{ iff } \up f(\up x) \cap (T \setminus V) = \emptyset\\
& \mbox{ iff } \up f(\up x)\cap T \subseteq V.
\end{align*}
By Remark~\ref{rem:p-morphism}(2), $\up f(\up x)=f(\up x)$. Therefore, $x \in f^*(j_T(V))$ iff $f(\up x) \cap T \subseteq V$. On the other hand,
\begin{align*}
x \in j_S(f^*(V)) & \mbox{ iff } x \in X \setminus \dn (S \setminus f^*(V))\\
& \mbox{ iff } \up x \cap S \subseteq f^*(V)\\
& \mbox{ iff } \up x \cap S \subseteq X \setminus \dn f^{-1}(Y \setminus V)\\
& \mbox{ iff } \up (\up x \cap S) \cap f^{-1}(Y \setminus V)= \emptyset\\
& \mbox{ iff } f(\up (\up x \cap S)) \subseteq V.
\end{align*}
Thus, $f^*(j_T(V))=j_S(f^*(V))$ for every $V \in \upsets(Y)$ iff for every $x \in X$ and $V \in \upsets(Y)$ we have $f(\up x) \cap T \subseteq V$ iff $f(\up (\up x \cap S)) \subseteq V$. Since $f(\up(\up x \cap S))$ is an upset by Remark~\ref{rem:p-morphism}(2), the latter condition is easily seen to be equivalent to Condition $(\ast)$ holding for every $x \in X$.
\end{proof}

\begin{lemma} \label{lem:S-morph}
Let $(X,S)$ and $(Y,T)$ be two finite S-posets, and let $f:X \to Y$ be a \kmor{} with domain $D \subseteq X$. Then Condition $(\ast)$ holds for every $x \in X$ iff the following two conditions hold:
\begin{enumerate}
\item $f^{-1}(T)=D \cap S$,
\item if $s \in S$, $d \in D$, and $s \le d$, then there are $s' \in S \cap D$ and $d' \in D$ such that $s\le s' \le d'$ and $f(d)=f(d')$.
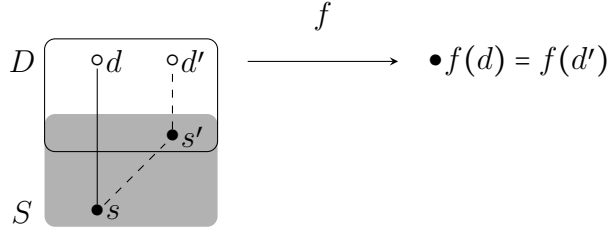
\begin{figure}[H]
\begin{tikzpicture}[>=stealth,inner sep=.0em,execute at end picture=
{\begin{pgfonlayer}{background}
\fill[gray!60!white,rounded corners] (-0.7,-2.2) rectangle (1.6,-.7);
\end{pgfonlayer}
}]
\node[label=right:$d$] (d) at (0,0) {$\circ$};
\node[label=right:$d'$] (d') at (1,0) {$\circ$};
\node[label=right:$s'$] (s') at (1,-1) {$\bullet$};
\node[label=right:$s$] (s) at (0,-2) {$\bullet$};
\draw (s) -- (d);
\draw[dashed] (s) -- (s');
\draw[dashed] (s') -- (d');
\draw[rounded corners] (-0.7,-1.2) rectangle (1.6,.3);
\node at (-1,0) {$D$};
\node at (-1,-2) {$S$};
\draw[->] (2,0) -- (4,0);
\node at (3,0.6) {$f$};
\node[label=right:${f(d)=f(d')}$] (fd) at (4.5,0) {$\bullet$};
\end{tikzpicture}
\caption{The second condition in the definition of an S-morphism.}\label{fig:S-morphism}
\end{figure}
\end{enumerate}
\end{lemma}

\begin{proof}
First suppose that $(\ast)$ holds for every
$x \in X$.
To see (1), let $x \in f^{-1}(T)$, so $x \in D$ and $f(x)\in T$. We have
\[
f(x) \in f(\up x)\cap T \subseteq \up (f(\up x) \cap T) =f(\up (\up x \cap S)).
\]
Therefore, there is $z \in \up (\up x \cap S)$ such that $f(x)=f(z)$. This implies that there is $w \in S$ such that $x \le w \le z$. Since $f$ preserves $<$ and $f(x)=f(z)$, we have $x=w=z$, so $x \in S$. This shows $f^{-1}(T) \subseteq D \cap S$. For the reverse inclusion, let $x \in D \cap S$. Then $x \in \up (\up x \cap S)$. Therefore, $f(x) \in f(\up (\up x \cap S))=\up (f(\up x) \cap T)$. Thus, there is
$y \in f(\up x) \cap T$ with
$y \le f(x)$. Since $x \in D$, by Remark~\ref{rem:p-morphism}(2), $f(\up x)=\up f(x)$. Therefore, $f(x) \le y \le f(x)$, yielding $f(x)=y$, which gives $x \in f^{-1}(T)$. This proves $D \cap S \subseteq f^{-1}(T)$, so (1) holds.

To see (2), let $s \in S$, $d \in D$, and $s \le d$. Then $d \in \up (\up s \cap S)$ which implies $f(d) \in f(\up (\up s \cap S))=\up (f(\up s) \cap T)$. Therefore, there is $t \in T$ such that $t \le f(d)$ and $t=f(s')$ for some $s' \ge s$. By (1), $s' \in f^{-1}(T)=S \cap D$. Since $f(s') \le f(d)$ and $f$ is a \kmor{}, there is $d' \in D$ such that $s' \le d'$ and $f(d')=f(d)$. Thus, (2) holds.

Conversely, we prove that (1) and (2) imply that $(\ast)$ holds for all $x \in X$. To see the left-to-right inclusion, let $y \in \up (f(\up x) \cap T)$. Then there is $t \in T$ such that $t \le y$ and $t \in f(\up x)$. So there is $s \in D$ such that $x \le s$ and $t =f(s)$. Since $s \in f^{-1}(T)$, we have $s \in S$ by (1). Also $f(s)=t \le y$ implies that there is $d \in D$ such that $s \le d$ and $f(d)=y$. Thus, $d \in \up (\up x \cap S)$, and so $y \in f(\up (\up x \cap S))$.

To see the right-to-left inclusion, let $y \in f(\up(\up x \cap S))$. Then there is $d \in D\cap \up (\up x \cap S)$ such that $y=f(d)$. Therefore, there is $s \in S$ such that $x \le s \le d$. By (2), there are $s' \in S\cap D$, $d' \in D$ such that $s \le s' \le d'$ and $f(d')=f(d)$. So $f(s') \le f(d')=f(d)=y$. Since $s' \in S\cap D$, (1) implies $f(s') \in T$. Thus, $f(s') \in f(\up x) \cap T$. Consequently, $y \in \up (f(\up x) \cap T)$.
\end{proof}

\begin{definition}
Let $(X,S)$ and $(Y,T)$ be two S-posets. We call a \kmor{} $f:X \to Y$ an \textit{S-morphism} if it satisfies the two conditions of Lemma~\ref{lem:S-morph}.
\end{definition}

\begin{remark}\label{rem:total S-morp conditions}
If $f:X \to Y$ is a total \kmor{}, then the second condition of Lemma~\ref{lem:S-morph} is trivially satisfied. Therefore, a total \kmor{} is an S-morphism iff $f^{-1}(T)=S$.
\end{remark}

It is easy to see that the identity morphism is an S-morphism. We next show that the composition of two S-morphisms is an S-morphism. This will imply that S-posets and S-morphisms form a category.

\begin{lemma}
The composition of two S-morphisms is an S-morphism.
\end{lemma}

\begin{proof}
Let $(X_1,S_1)$, $(X_2,S_2)$, $(X_3,S_3)$ be S-posets and let $f_1:X_1 \to X_2$, $f_2:X_2 \to X_3$ be S-morphisms with domains $D_1$ and $D_2$. Then $f_2  f_1$ is a \kmor{} with domain $D_3=f_1^{-1}(D_2)$. We show that the two conditions of Lemma~\ref{lem:S-morph} are satisfied. For the first condition, since $f_1$ and $f_2$ are S-morphisms, we have
\begin{align*}
(f_2 f_1)^{-1}(S_3) & = f_1^{-1}(f_2^{-1}(S_3)) \\
& = f_1^{-1}(D_2 \cap S_2) \\
& = f_1^{-1}(D_2) \cap f_1^{-1}(S_2)\\
& = D_3 \cap D_1 \cap S_1 = D_3 \cap S_1.
\end{align*}
Thus, the first condition is satisfied. For the second condition, let $s \in S_1$, $d \in D_3$, and $s \le d$. Since $d \in D_3 \subseteq D_1$, $s \in S_1$, and $f_1$ is an S-morphism, there are $s_1 \in D_1 \cap S_1$ and $d_1 \in D_1$ such that $s \le s_1 \le d_1$ and $f_1(d)=f_1(d_1)$. Since $f_1^{-1}(S_2)=D_1 \cap S_1$, we have $f_1(s_1) \in S_2$. From $d \in D_3=f_1^{-1}(D_2)$ it follows that $f_1(d_1)=f_1(d) \in D_2$. Since $f_1$ is order-preserving, $f_1(s_1) \le f_1(d_1)$. Because $f_2$ is an S-morphism, there are $s_2 \in S_2 \cap D_2$ and $d_2 \in D_2$ such that $f_1(s_1) \le s_2 \le d_2$ and $f_2(d_2)=f_2f_1(d_1)=f_2f_1(d)$. Since $f_1$ is a \kmor{} and $f_1(s_1) \le s_2$, there is $s_3 \in D_1$ such that $s_1 \le s_3$ and $f_1(s_3)=s_2$. We have $f_1(s_3)=s_2 \in S_2 \cap D_2$, so $s_3 \in f_1^{-1}(S_2)=D_1 \cap S_1$ and $s_3 \in f_1^{-1}(D_2)=D_3$. Thus, $s_3 \in S_1 \cap D_3$. From $f_1(s_3)=s_2 \le d_2$ it follows that there is $d_3 \in D_1$ such that $s_3 \le d_3$ and $f_1(d_3)=d_2$. Since $d_2 \in D_2$, we have $d_3 \in f_1^{-1}(D_2)=D_3$ and $f_2f_1(d_3)=f_2(d_2)=f_2f_1(d)$. We also have $s \le s_1 \le s_3$, so taking $s'=s_3$ and $d'=d_3$ yields $s \le s' \le d'$ with $s' \in S_1 \cap D_3$, $d' \in D_3$, and $f_2  f_1(d)=f_2  f_1(d')$. Thus, $f_2  f_1$ is an S-morphism.
\end{proof}

\begin{definition}
Let $\spf$ be the category of finite S-posets and S-morphisms.
\end{definition}

\begin{remark} \label{rem:iso in spf}
Let $f:X \to Y$ be a map between finite S-posets $(X,S)$ and $(Y,T)$. If $f$ is a poset isomorphism between $X$ and $Y$,
then it follows from Remark~\ref{rem:total S-morp conditions} that
$f$ is an S-morphism iff $f(S)=T$. Therefore, $f$ is an isomorphism in $\spf$ iff $f$ is a poset isomorphism and $f(S)=T$.
\end{remark}

We next define contravariant functors between $\nisf$ and $\spf$. The functor $(\;)^*:\spf \to \nisf$ associates with each finite S-poset $\mathfrak{X}=(X,S)$ the finite nuclear implicative semilattice $\mathfrak{X}^*=(\upsets(X),j_S)$; and with each S-morphism $f:\mathfrak{X} \to \mathfrak{Y}$ the nuclear implicative semilattice homomorphism $f^*:\mathfrak{Y}^* \to \mathfrak{X}^*$ given by $f^*(V)=X \setminus \dn (Y \setminus V)$ for each $V \in \upsets(Y)$.
That $(\;)^*$ is well defined
follows from Lemmas~\ref{prop:S-morph} and~\ref{lem:S-morph}. That $(\;)^*$ preserves identities and reverses compositions follows from K\"ohler duality. Thus, $(\;)^*$ is a well-defined contravariant functor.

To define the functor $(\;)_*:\nisf \to \spf$ we require the following two lemmas.

\begin{lemma}\label{lem:meet-prime comp of hh_*(x)}
Let $A$ and $B$ be two finite implicative semilattices and $h:A \to B$ an implicative semilattice homomorphism. For any $x$ in the domain of $h_*:X_B \to X_A$, the set $h_*^{-1}h_*(x)$ coincides with the set of the meet-prime components of $hh_*(x)$ in $B$. In particular, $x$ is a meet-prime component of $hh_*(x)$.
\end{lemma}

\begin{proof}
Let $l$ be the nucleus on $B$ corresponding to the subalgebra $h(A)$. Then $S_l=X_B\cap B_l$ is the domain of $h_*$. Let $x\in S_l$. First suppose that $z$ is a meet-prime component of $hh_*(x)$ in $B$. Since $hh_*(x) \in h(A) \subseteq B_l$, Lemma~\ref{lem:components of fixpoints} yields that $z \in S_l$. Therefore, $hh_*(x) \le z$ implies $h_*(x) \le h_*(z)$. Since $h_*:X_B \to X_A$ is a \kmor{} and the orders on $X_B$ and $X_A$ are duals of the orders on $A$ and $B$, from $h_*(x) \le h_*(z)$ it follows that there is $y \in S_l$ such that $y \le z$ and $h_*(y)=h_*(x)$. As $y \in S_l$, we have $hh_*(x)=hh_*(y) \le y \le z$.
Since $y$ is a meet-prime, by minimality of $z$,
we have $y=z$. Therefore, $h_*(z)=h_*(x)$, and so $z \in h_*^{-1}h_*(x)$.

Conversely, suppose that $z \in h_*^{-1}h_*(x)$, so $h_*(x)=h_*(z)$. Since $z \in S_l$, we have $hh_*(x) \le z$. Because $z$ is meet-prime, there is a meet-prime component $z'$ of $hh_*(x)$ such that $hh_*(x)\le z' \le z$. Since $hh_*(x) \in h(A) \subseteq B_l$, we have $z' \in S_l$ by Lemma~\ref{lem:components of fixpoints}, so $h_*(x) \le h_*(z')$,
and hence $h_*(z) \le h_*(z')$.
As $h_*$ is order-preserving, $h_*(z') \le h_*(z)$. Therefore, $h_*(z')=h_*(z)$. Since $z' \le z$ and $h_*$ is a \kmor{}, $z'=z$. Thus, $z$ is a meet-prime component of $hh_*(x)$.

Finally, since $x \in h_*^{-1}h_*(x)$, it follows that $x$ is a meet-prime component of $hh_*(x)$.
\end{proof}

\begin{lemma}
Let $\mathfrak{A}=(A,j)$ and $\mathfrak{B}=(B,k)$ be two finite nuclear implicative semilattices. If $h:\mathfrak{A} \to \mathfrak{B}$ is a nuclear implicative semilattice homomorphism, then $h_*:\mathfrak{B}_* \to \mathfrak{A}_*$ is an S-morphism.
\end{lemma}

\begin{proof}
By Lemma~\ref{lem:kohler morph},
$h_*:X_B \to X_A$ is a \kmor{} with domain $S_l$ where $l$ is the nucleus on $B$ corresponding to the subalgebra $h(A)$. It remains to verify the two conditions of Lemma~\ref{lem:S-morph}. We first prove the first condition that $h_*^{-1}(S_j)=S_l \cap S_k$. For the left-to-right inclusion, let $x \in h_*^{-1}(S_j)$. Then $x \in S_l$ and $h_*(x) \in S_j$. Since $jh_*(x)=h_*(x)$ and $h$ is a nuclear implicative semilattice homomorphism, we have
\[
khh_*(x)=hjh_*(x)=hh_*(x).
\]
Therefore, $hh_*(x) \in B_k$. By Lemma~\ref{lem:meet-prime comp of hh_*(x)}, $x$ is a meet-prime component of $hh_*(x)$. So Lemma~\ref{lem:components of fixpoints} implies that $x \in B_k$. Thus, $x \in S_l \cap S_k$. For the right-to-left inclusion, let $x \in S_l \cap S_k$. Since $x \in S_l$, we have $hh_*(x) \le x$. That $h$ is a nuclear implicative semilattice homomorphism then implies
\[
hjh_*(x)=khh_*(x) \le k(x)=x.
\]
Therefore, $jh_*(x) \le h_*(x)$ which yields $h_*(x) \in S_j$. Thus, $x \in h_*^{-1}(S_j)$.

We next prove the second condition of Lemma~\ref{lem:S-morph}. Recalling that the order on $X_B$ is dual to the order of $B$, let $s \in S_k$, $d \in S_l$, and $d \le s$. We have
\[
khh_*(d) \le k(d) \le k(s) =s.
\]
Since $s$ is meet-prime, there is a meet-prime component $s'$ of $khh_*(d)$ such that $s' \le s$. Note that $khh_*(d) \in B_k$ and $khh_*(d)=hjh_*(d) \in h(A) \subseteq B_l$. Therefore, by Lemma~\ref{lem:components of fixpoints}, $s' \in S_k \cap S_l$. Since $hh_*(d) \le khh_*(d) \le s'$ and $s'$ is meet-prime, there is a meet-prime component $d'$ of $hh_*(d)$ such that $d' \le s'$. Since $hh_*(d) \in h(A) \subseteq B_l$, Lemma~\ref{lem:components of fixpoints} implies that $d' \in S_l$. Also, $d' \in h_*^{-1}h_*(d)$ by Lemma~\ref{lem:meet-prime comp of hh_*(x)}. Thus, $h_*(d)=h_*(d')$.
\end{proof}

We are ready to define the functor $(\;)_*:\nisf \to \spf$ which associates with each finite nuclear implicative semilattice $\mathfrak{A}=(A,j)$ the S-poset $\mathfrak{A}_*=(X_A,S_j)$; and with each nuclear implicative semilattice homomorphism $h:\mathfrak{A} \to \mathfrak{B}$ the S-morphism $h_*:\mathfrak{B}_* \to \mathfrak{A}_*$.
That $(\;)_*$ preserves identities and reverses compositions is immediate from K\"ohler duality, thus $(\;)_*$ is a contravariant functor.

It follows from Theorem~\ref{thm: (A,j) iso to (A,j)_*^*} that for every $\mathfrak{A}=(A, j)\in \nisf$ the map $\alpha_A:\mathfrak{A} \to (\mathfrak{A}_*)^*$ is an isomorphism in $\nisf$. We next show that for every $\mathfrak{X}=(X,S) \in \spf$ the map $\varepsilon_X:\mathfrak{X} \to (\mathfrak{X}^*)_*$ is an isomorphism in $\spf$.

\begin{lemma} \label{lem:epsilon S-iso}
If $(X,S)$ is a finite S-poset, then $\varepsilon_X:X \to X_{\upsets(X)}$ is an S-poset isomorphism between $(X,S)$ and $(X_{\upsets(X)}, S_j)$.
\end{lemma}

\begin{proof}
We recall that the isomorphism $\varepsilon_X:X \to X_{\upsets(X)}$ is given by $\varepsilon_X(x)=X \setminus \dn x$. By Remark~\ref{rem:iso in spf}, it is sufficient to show that $\varepsilon_X(S)=S_{j_S}$. We have
\begin{align*}
S_{j_S} &= \{ X \setminus \dn x \mid j_S(X \setminus \dn x)=X \setminus \dn x \}\\
&=\{ X \setminus \dn x \mid X \setminus \dn( S \setminus (X \setminus \dn x))=X \setminus \dn x \}\\
&=\{ X \setminus \dn x \mid \dn( S \setminus (X \setminus \dn x))=\dn x \}\\
&=\{ X \setminus \dn x \mid \dn( S \cap \dn x)=\dn x \}.
\end{align*}
Since $\dn( S \cap \dn x)=\dn x$ iff $x\in S$, we have
\[
S_{j_S} = \{ X \setminus \dn x \mid x \in S \} = \varepsilon_X(S).
\]
\end{proof}

Consequently, $\alpha_A$ and $\varepsilon_X$ are natural isomorphisms, so the functors $(\;)^*:\spf \to \nisf$ and $(\;)_*:\nisf \to \spf$ yield a dual equivalence, and we arrive at the following generalization of K\"ohler duality.

\begin{theorem}\label{thm:duality for nisf}
$\nisf$ is dually equivalent to $\spf$.
\end{theorem}

\section{Dual description of subalgebras}\label{sec:subalgebras}

We next would like to utilize Theorem~\ref{thm:duality for nisf} to give a dual description of finitely generated finite nuclear implicative semilattices. For this we require a dual description of subalgebras, which is the subject of this section.

We start by recalling from \cite[Lem.~3.4]{Koh81} that one-to-one morphisms in $\isf$ correspond to onto morphisms in $\pfk$, and that onto morphisms in $\isf$ correspond to one-to-one morphisms in $\pfk$. This result directly generalizes to the setting of $\nisf$ and $\spf$:

\begin{proposition}\label{prop:one-to-one and onto dual correspondence}
Let $h:A \to B$ be a nuclear homomorphism between finite nuclear implicative semilattices and let $h_*:X_B \to X_A$ be its dual S-morphism.
\begin{enumerate}
\item $h$ is one-to-one iff $h_*$ is onto,
\item $h$ is onto iff $h_*$ is total and one-to-one.
\end{enumerate}
\end{proposition}

Since images of total and one-to-one S-morphisms are upsets of the target, as an immediate consequence of Proposition~\ref{prop:one-to-one and onto dual correspondence}(2), we obtain:

\begin{proposition}
Let $(A,j)$ be a finite nuclear implicative semilattice. Homomorphic images of $(A,j)$ dually correspond to upsets of $X_A$.
\end{proposition}

By Proposition~\ref{prop:one-to-one and onto dual correspondence}(1), nuclear subalgebras of a finite nuclear implicative semilattice dually correspond to onto S-morphisms. Each such gives rise to a partial equivalence relation. To characterize these, we recall
that Heyting subalgebras of a finite Heyting algebra $A$ dually correspond to correct partitions of $X_A$.

\begin{definition}
A \emph{correct partition} of a poset $(X,\le)$ is an equivalence relation $\sim$ on $X$ such that $x \sim y$ and $y \le z$ imply that there is $w\in X$ such that $x \le w$ and $w \sim z$.
\end{definition}

\begin{definition}
Let $(X,\le)$ be a poset. A \textit{$($strict$)$ partial correct partition} of $X$ is an equivalence relation $\sim$ on $D \subseteq X$
(which we call the domain of $\sim$)
such that
\begin{enumerate}
\item all equivalence classes of $\sim$ are antichains;
\item $x \sim y$, $y < z$, and $z \in D$ imply that there is $w \in D$ such that $x < w$ and $w \sim z$.
\end{enumerate}
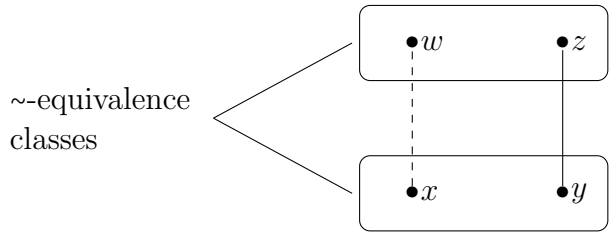
\begin{figure}[H]
\begin{tikzpicture}[>=stealth,inner sep=.0em]
\node[label=right:$w$] (w) at (0,0) {$\bullet$};
\node[label=right:$z$] (z) at (2,0) {$\bullet$};
\node[label=right:$y$] (y) at (2,-2) {$\bullet$};
\node[label=right:$x$] (x) at (0,-2) {$\bullet$};
\draw[dashed] (x) -- (w);
\draw (y) -- (z);
\draw[rounded corners] (-0.7,-2.5) rectangle (2.6,-1.5);
\draw[rounded corners] (-0.7,-0.5) rectangle (2.6,.5);
\node[text width=2.7cm] (stext) at (-4,-1) {$\sim$-equivalence classes};
\draw (stext.east) -- (-0.8,0);
\draw (stext.east) -- (-0.8,-2);
\node at (7,0) {$ $};
\end{tikzpicture}
\caption{The second condition in the definition of a partial correct partition.}\label{fig:partial correct partition}
\end{figure}
\end{definition}

\begin{proposition} \label{prop:subalgebras of U(X) and correct partitions}
Let $(X,\le)$ be a finite poset. Subalgebras of $\upsets(X)$ dually correspond to partial correct partitions of $X$.
\end{proposition}

\begin{proof}
By Proposition~\ref{prop:one-to-one and onto dual correspondence}(1), subalgebras of $\upsets(X)$ dually correspond to onto \kmor{}s on $X_{\upsets(X)}$. These correspond to partial correct partitions of $X_{\upsets(X)}$. Indeed, if $\pi:X_{\upsets(X)}\to Y$ is an onto \kmor{}, then $\sim$ given by $x\sim y$ iff $\pi(x)=\pi(y)$ is a partial correct partition of $X_{\upsets(X)}$. Conversely, for each partial correct partition $\sim$ of $X_{\upsets(X)}$, let $Y$ be the quotient of the domain of $\sim$, and let $\pi:X_{\upsets(X)}\to Y$ be the partial quotient map. Then $\pi$ is an onto \kmor{} whose corresponding partial correct partition is $\sim$.
Since $\varepsilon_X:X\to X_{\upsets(X)}$ is an isomorphism, partial correct partitions of $X$ correspond to partial correct partitions of $X_{\upsets(X)}$. Thus, subalgebras of $\upsets(X)$ dually correspond to partial correct partitions of $X$.
\end{proof}

Proposition~\ref{prop:subalgebras of U(X) and correct partitions} together with Lemmas~\ref{prop:S-morph} and~\ref{lem:S-morph} yields the following dual characterization of nuclear subalgebras.
We say that a subset $U$ of the domain $D$ of a partial correct partition $\sim$ is \textit{saturated} provided $x \in U$ and $x \sim y$ imply $y \in U$.

\begin{proposition} \label{prop:dual sub}
Let $(X,S)$ be a finite S-poset. Subalgebras of $(\upsets(X), j_S)$ dually correspond to partial correct partitions $\sim$ of $X$ with domain $D$ such that
\begin{enumerate}
\item $S \cap D$ is saturated,
\item if $s \in S$, $d \in D$, and $s \le d$, then there are $s' \in S \cap D$ and $d' \in D$ such that $s\le s' \le d'$ and $d \sim d'$.
\end{enumerate}
\end{proposition}
\begin{figure}[H]
\begin{tikzpicture}[>=stealth,inner sep=.0em,execute at end picture=
{\begin{pgfonlayer}{background}
\fill[gray!60!white,rounded corners] (.5,-1.3) rectangle (2,-.7);
\end{pgfonlayer}
}]
\node[label=right:$d$] (d) at (0,0) {$\circ$};
\node[label=right:$d'$] (d') at (1,0) {$\circ$};
\node[label=right:$s'$] (s') at (1,-1) {$\bullet$};
\node[label=right:$s$] (s) at (0,-2) {$\bullet$};
\node[text width=3.2cm] (nostext) at (-5,.5) {$\sim$-equivalence classes outside $S$};
\node[text width=2.7cm] (stext) at (5,0) {$\sim$-equivalence classes in $S$};
\draw (s) -- (d);
\draw[dashed] (s) -- (s');
\draw[dashed] (s') -- (d');
\draw[rounded corners] (-1,-.3) rectangle (1.5,.3);
\draw[rounded corners] (-2.4,-1.1) rectangle (-.2,-.5);
\draw (nostext.south east) -- (-1.2,0);
\draw (nostext.south east) -- (-2.5,-.4);
\draw (stext.west) -- (2.2,-1);
\draw (stext.west) -- (2.4,.7);
\fill[gray!60!white,rounded corners] (-1.2,.5) rectangle (2.2,1.1);
\end{tikzpicture}
\caption{Second condition in the definition of a nuclear partial correct partition.}\label{fig:nuclear partial correct partition}
\end{figure}

\begin{definition}
We call the partial correct partitions satisfying the two conditions of Proposition~\ref{prop:dual sub} \textit{nuclear partial correct partitions}.
\end{definition}

We next provide a decomposition of homomorphisms of finite implicative semilattices, and show that the corresponding decomposition does not hold for finite nuclear implicative semilattices.

\begin{definition} \label{def:strict Heyting subalg}
Let $(X,\le)$ be a finite poset. We say that a partial correct partition of $X$ is \textit{total} if its domain is $X$.
We call a subalgebra $\mathfrak{B}$ of $\upsets(X)$ a \textit{strict Heyting subalgebra} if it corresponds to a total correct partition of $X$.
\end{definition}

\begin{remark}
Let $(X,\le)$ be a finite poset, $\mathfrak{B}$ a strict Heyting subalgebra of $\upsets(X)$, and $\sim$ the corresponding total correct partition of $X$. If $\pi:X\to X/{\sim}$ is the corresponding quotient map, then $\pi^*=\pi^{-1}$, so $\pi^*:\upsets(X/{\sim})\to\upsets(X)$ is a Heyting algebra
embedding.
Since $\mathfrak B$ is isomorphic to $\upsets(X/{\sim})$, each strict Heyting subalgebra $\mathfrak B$ of $\upsets(X)$ is a Heyting subalgebra of $\upsets(X)$. We call it strict because the corresponding p-morphism is strict in that $x<y$ implies $\pi(x)<\pi(y)$.
\end{remark}

\begin{lemma}\label{lem:total and strict hey dually}
Let $(X,\le)$ be a finite poset, $\mathfrak B$ a subalgebra of $\upsets(X)$, $\sim$ the corresponding partial correct partition of $X$ with domain $D$, and $j$ the nucleus on $\upsets(X)$ induced by $\mathfrak B$.
\begin{enumerate}
\item $\mathfrak{B}$ is a total subalgebra of $\upsets(X)$ iff $\sim$ is the identity on its domain,
\item $\mathfrak{B}$ is a strict Heyting subalgebra of $\upsets(X)$ iff $j$ is the identity nucleus on $\upsets(X)$,
\item $\mathfrak{B}$ is a strict Heyting subalgebra of $\upsets(X)_j$.
\end{enumerate}
\end{lemma}

\begin{proof}
(1) Let $i:\mathfrak B\to\upsets(X)$ be the embedding, and let $i_*:X_{\upsets(X)}\to X_\mathfrak B$ be its dual. Then the domain of $i_*$ is $S_j$. By Proposition~\ref{prop:nucleus induced by subalgebra}, $\mathfrak{B}$ is a total subalgebra iff $\mathfrak B=\upsets(X)_j$. We have that $\mathfrak B=\upsets(X)_j$ iff $i_*$ is the identity on $S_j$ which is equivalent to $\sim$ being the identity on $D$ by Proposition~\ref{prop:subalgebras of U(X) and correct partitions}.

(2) $\mathfrak{B}$ is a strict Heyting subalgebra of $\upsets(X)$ iff $D=X$. By Proposition~\ref{prop:subalgebras of U(X) and correct partitions}, $\varepsilon_X(D)=S_j$. Therefore, $D=X$ iff $\varepsilon_X(D)=\varepsilon_X(X)$, which happens iff $S_j=X_{\upsets(X)}$. We have that $S_j=X_{\upsets(X)}$ iff $X_{\upsets(X)} \subseteq \upsets(X)_j$, which is equivalent to $j$ being the identity on $\upsets(X)$ because $\upsets(X)$ is generated by $X_{\upsets(X)}$.

(3) The nucleus on $\upsets(X)_j$ induced by $\mathfrak{B}$ is the restriction of $j$ to $\upsets(X)_j$, so it is the identity on $\upsets(X)_j$. Therefore, by (2), $\mathfrak{B}$ is a strict Heyting subalgebra of $\upsets(X)_j$.
\end{proof}

\begin{remark}
We can restate Lemma~\ref{lem:total and strict hey dually} as follows:

Let $A$ be a finite implicative semilattice, $B$ a subalgebra of $A$, $\sim$ the corresponding partial correct partition of $X_A$,
and $j$ the nucleus on $A$ induced by $B$.
\begin{enumerate}
\item $B$ is a total subalgebra of $A$ iff $\sim$ is the identity on its domain,
\item $B$ is a strict Heyting subalgebra of $A$ iff $j$ is the identity nucleus on $A$,
\item $B$ is a strict Heyting subalgebra of $A_j$.
\end{enumerate}
\end{remark}

From these characterizations of total subalgebras and strict Heyting subalgebras we obtain the following decomposition of every morphism in $\isf$ and every morphism in $\pfk$.

\begin{proposition}\label{prop:decomposition}
Let $h:A \to B$ be an implicative semilattice homomorphism between finite implicative semilattices, and let $j$ be the nucleus on $B$ induced by its subalgebra $h(A)$. Then $h$ can be written as the composition of the homomorphisms:
\begin{enumerate}
\item the onto homomorphism $h_1:A \to h(A)$ obtained by restricting the codomain of $h$,
\item the inclusion $h_2$ of the strict Heyting subalgebra $h(A)$ into $B_j$, and
\item the inclusion $h_3$ of the total subalgebra $B_j$ into $B$.
\end{enumerate}
\[
\begin{tikzcd}[column sep=5pc]
A \arrow[d, "h_1"'] \arrow[r, "h"] & B \\
h(A) \arrow[r, "h_2"'] & B_j \arrow[u, "h_3"']
\end{tikzcd}
\]
Let $f:X \to Y$ be a \kmor{} between finite posets with domain $D \subseteq X$. Then $f$ can be written as the composition of the \kmor{}s:
\begin{enumerate}
\item the onto \kmor{} $f_1:X \to D$ which is the identity on its domain $D$,
\item the total onto \kmor{} $f_2:D \to f(X)$ obtained by restricting the codomain of $f$,
\item the inclusion $f_3$ of $f(X)$ into $Y$ as an upset.
\end{enumerate}
\[
\begin{tikzcd}[column sep=5pc]
X \arrow[d, "f_1"'] \arrow[r, "f"] & Y \\
D \arrow[r, "f_2"'] & f(X) \arrow[u, "f_3"']
\end{tikzcd}
\]
\end{proposition}

\begin{definition}
Let $\mathfrak{A}$ be a finite nuclear implicative semilattice and let $\mathfrak{B}$ be a nuclear subalgebra of $\mathfrak{A}$.
\begin{enumerate}
\item We call $\mathfrak{B}$ a \textit{total nuclear subalgebra} of $\mathfrak{A}$ if $\mathfrak{B}$ is a total subalgebra of $\mathfrak{A}$.
\item We call $\mathfrak{B}$ a \textit{strict Heyting nuclear subalgebra} of $\mathfrak{A}$ if $\mathfrak{B}$ is a strict Heyting subalgebra of $\mathfrak{A}$.
\end{enumerate}
\end{definition}

One would expect that Proposition~\ref{prop:decomposition} generalizes to the nuclear setting. The next example shows that this is not so.

\begin{example}
Let $\mathfrak{X}=(X,S)$ and $\mathfrak{Y}=(Y,T)$ be the S-posets shown on the left and right of Figure~\ref{fig:counterexample decomposition}. Let $f:\mathfrak{X} \to \mathfrak{Y}$ be the onto S-morphism with domain $D$ whose decomposition is shown in Figure~\ref{fig:counterexample decomposition}. Since $f$ is onto, $f_3$ is the total identity map, so we can ignore $f_3$.
The \kmor{} $f_1:X \to D$ cannot be an S-morphism for any S-poset structure that we put on $D$ because it does not satisfy the second condition of Lemma~\ref{lem:S-morph}. Equivalently, the partial correct partition of $X$ corresponding to $f_1$ is not a nuclear partial correct partition of $\mathfrak{X}$ since it does not satisfy the second condition of Proposition~\ref{prop:dual sub}. Dually this means that if $\mathfrak{A}=(A,j)$ and $\mathfrak{B}=(B,k)$ are finite nuclear implicative semilattices, $h:\mathfrak{A} \to \mathfrak{B}$ is a homomorphism and $l$ is the nucleus on $B$ induced by $h(A)$, then $B_l$ is not necessarily a nuclear subalgebra of $\mathfrak{B}$.

\begin{figure}[H]
\begin{tikzpicture}[>=stealth,inner sep=.0em,execute at end picture=
{\begin{pgfonlayer}{background}
\fill[gray!60!white,rounded corners] (-0.7,-2.2) rectangle (1.6,-.7);
\end{pgfonlayer}
}]
\node[label=right:$d$] (d) at (0,0) {$\circ$};
\node[label=right:$d'$] (d') at (1,0) {$\circ$};
\node[label=right:$s'$] (s') at (1,-1) {$\bullet$};
\node[label=right:$s$] (s) at (0,-2) {$\bullet$};
\draw (s) -- (d);
\draw (s) -- (s');
\draw (s') -- (d');
\draw[rounded corners] (-0.7,-1.5) rectangle (1.6,.3);
\node at (-1,0.2) {$D$};
\node at (-1,-2) {$S$};
\node at (0.5,1) {$X$};
\draw[->] (2,-0.5) -- (3.8,-0.5);
\node at (3,-0.2) {$f_1$};
\node[label=right:$d$] (d1) at (0+5,0) {$\circ$};
\node[label=right:$d'$] (d1') at (1+5,0) {$\circ$};
\node[label=right:$s'$] (s1') at (1+5,-1) {$\bullet$};
\draw (s1') -- (d1');
\draw[rounded corners] (-0.7+5,-1.5) rectangle (1.6+5,.3);
\node at (0.5+5,1) {$D$};
\draw[->] (2+5,-0.5) -- (3.8+5,-0.5);
\node at (3+5,-0.2) {$f_2$};
{\begin{pgfonlayer}{background}
\fill[gray!60!white,rounded corners] (-0.7+8.5+1.2,-1.5) rectangle (1.6+8.5+0.7,-.7);
\end{pgfonlayer}
}]
\node[label=right:${f(d)=f(d')}$] (d2') at (1+8.5,0) {$\circ$};
\node[label=right:$f(s')$] (s2') at (1+8.5,-1) {$\bullet$};
\draw (s2') -- (d2');
\node at (-1+5+6,-1.8) {$T$};
\node at (1+8.5,1) {$Y$};
\end{tikzpicture}
\caption{S-morphism whose decomposition is not made of S-morphisms.}\label{fig:counterexample decomposition}
\end{figure}
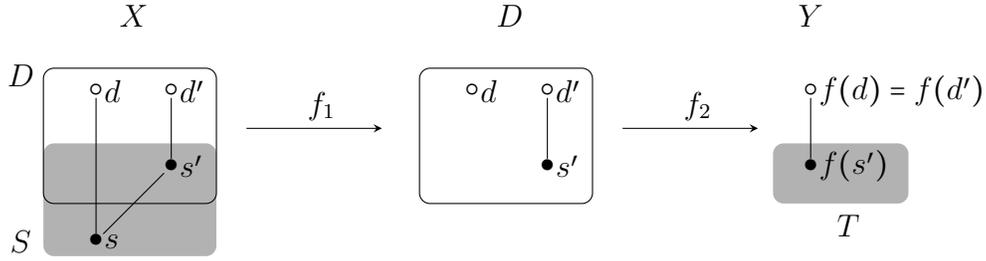
\end{example}

We next characterize dually nuclear subalgebras that are total and strict Heyting.

\begin{proposition}\label{prop:dual total and strict Hey nuclear subalg}
Let $(X,S)$ be a finite S-poset and let $\sim$ be a partial correct partition of $X$ with domain $D$. Then $\sim$ corresponds to a total nuclear subalgebra of $(\upsets(X), j_S)$ iff
\begin{enumerate}
\item $\sim$ is the identity relation on $D$,
\item for all $d \in D$, we have $\max (S \cap \dn d) \subseteq D$.
\end{enumerate}
Moreover, $\sim$ corresponds to a strict Heyting nuclear subalgebra of $(\upsets(X), j_S)$ iff
\begin{enumerate}
\item $D=X$, i.e.~$\sim$ is a total correct partition,
\item $S$ is saturated.
\end{enumerate}
\end{proposition}

\begin{proof}
By Lemma~\ref{lem:total and strict hey dually}, $\sim$ corresponds to a total subalgebra iff it is the identity relation on $D$ and it corresponds to a strict Heyting subalgebra iff $D=X$. It only remains to show how the conditions defining nuclear partial correct partitions simplify in these two special cases.

First, suppose that $\sim$ corresponds to a total subalgebra. Then $\sim$ is the identity on $D$. Therefore, $S \cap D$ is always saturated. We show that the second condition in Lemma~\ref{lem:S-morph} simplifies as above. If $\sim$ is a nuclear partial correct partition, $d \in D$ and $s \in \max (S \cap \dn d)$, then there are $s' \in S \cap D$ and $d' \in D$ such that $s\le s' \le d'$ and $d \sim d'$. But $\sim$ is the identity on $D$, so $d=d'$ which gives $s \le s' \in S \cap \dn d$. The maximality of $s$ then implies that $s=s'$, so $s \in D$. Conversely, suppose that $\max (S \cap \dn d) \subseteq D$ for all $d \in D$. Let $s \in S$, $d\in D$, and $s \le d$. Then $s \in S \cap \dn d$. Since $X$ is finite, there is $s' \in \max(S \cap \dn d)$ such that $s \le s'$. By our assumption, $s' \in S \cap D$, and we clearly have that $s \le s' \le d$.

Next, suppose that $\sim$ corresponds to a strict Heyting subalgebra. Then $D=X$. Therefore, the second condition in Lemma~\ref{lem:S-morph} trivially holds. Thus, we only have to require that $S \cap D=S$ is saturated.
\end{proof}

We conclude this section by characterizing dually maximal subalgebras and maximal nuclear subalgebras.

\begin{definition}
We call a proper subalgebra $B$ of an implicative semilattice $A$ \textit{maximal} if there is no proper subalgebra of $A$ properly containing $B$.
We call a proper nuclear subalgebra of a nuclear implicative semilattice \textit{maximal} if it is maximal among the proper nuclear subalgebras.
\end{definition}

To characterize maximal subalgebras, we require the following lemma.

\begin{lemma}\label{lemma:inclusion subalg dual}
Let $(X, \le)$ be a finite poset, $\mathfrak B_1, \mathfrak B_2$ subalgebras of $\upsets(X)$, and $\sim_1,\sim_2$ the corresponding partial correct partitions of $X$ with domains $D_1,D_2 \subseteq X$. Then $\mathfrak B_1 \subseteq \mathfrak B_2$ iff the following conditions are satisfied:
\begin{enumerate}
\item $D_1 \subseteq D_2$,
\item $D_1$ is saturated with respect to $\sim_2$,
\item $\sim_1$ is an extension of $\sim_2$ on $D_1$.
\end{enumerate}
\end{lemma}

\begin{proof}
Let $Y_1:=X/{\sim_1}$ and $Y_2:=X/{\sim_1}$ be the quotients and $\pi_1:X \to Y_1$ and $\pi_2:X \to Y_2$ the corresponding \kmor{}s. The domain of each $\pi_i$ is $D_i$, and $x \sim_i y$ iff $\pi_i(x)=\pi_i(y)$.
By Proposition~\ref{prop:subalgebras of U(X) and correct partitions},
$\mathfrak B_1 \subseteq \mathfrak B_2$
iff there is an onto \kmor{} $\pi_3: Y_2 \to Y_1$ such that $\pi_3 \pi_2=\pi_1$.

Now suppose that $\mathfrak B_1 \subseteq \mathfrak B_2$. Let $D_3 \subseteq Y_2$ be the domain of $\pi_3$. Then
\[
D_1=\pi_1^{-1}(Y_1)=\pi_2^{-1}\pi_3^{-1}(Y_1)=\pi_2^{-1}(D_3).
\]
So $D_1 \subseteq D_2$ and $D_1$ is saturated with respect to $\sim_2$. Moreover, if $x,y \in D_1$ with $x \sim_2 y$, then $\pi_1(x)=\pi_3 \pi_2(x)=\pi_3 \pi_2(y)=\pi_1(y)$, so $x \sim_1 y$.

Conversely, suppose that the three conditions are satisfied. Define $\pi_3:Y_2 \to Y_1$ by setting the domain of $\pi_3$ to be $\pi_2(D_1)$ and
$\pi_3(\pi_2(x))=\pi_1(x)$. We have that $\pi_3$ is well defined because $D_1 \subseteq D_2$ and $\sim_1$ extends $\sim_2$ on $D_1$. To see that $\pi_3$ is strict, let $x,y \in D_1$ be such that $\pi_2(x) < \pi_2(y)$. Then there is $z \in D_2$ with $x < z$ and $\pi_2(z)=\pi_2(y)$. Therefore, $z \sim_2 y$. Since $D_1$ is saturated
with respect to $\sim_2$,
$z \in D_1$, so $\pi_3(\pi_2(x))=\pi_1(x) < \pi_1(z)=\pi_3(\pi_2(z))=\pi_3(\pi_2(y))$. It is then straightforward to see that $\pi_3$ is a \kmor{} because
$\pi_1$ and $\pi_2$ are \kmor{}s.
From the definition of $\pi_3$ it is clear that $\pi_3 \pi_2=\pi_1$. Thus, $\mathfrak B_1\subseteq \mathfrak B_2$.
\end{proof}

\begin{definition}
Let $(X,\le)$ be a finite poset and $x\in X$.
\begin{enumerate}
\item We call $y$ an \emph{$($upper$)$ cover} of $x$ if $x<y$ and there is no $z\in X$ with $x<z<y$.
\item Let $\coverof x$ be the set of covers of $x$.
\end{enumerate}
\end{definition}

\begin{theorem} \label{thm:dual maximal subalg}
Let $(X,\le)$ be a finite poset. Maximal subalgebras of $\upsets(X)$ correspond to partial correct partitions $\sim$ of $X$ with domain $D$ such that either
\begin{enumerate}
\item $D=X \setminus \{ x \}$ for some $x \in X$ and $\sim$ is the identity relation on $D$, or
\item $D=X$ and the only non-trivial equivalence class of $\sim$ is $\{x,y\}$ with $\coverof x=\coverof y$.
\end{enumerate}
Maximal subalgebras of the first kind are total, and the ones of the second kind are strict Heyting.
\end{theorem}

\begin{proof}
It is straightforward to see that the partial equivalence relations defined in (1) and (2) are partial correct partitions. It follows from Lemma~\ref{lemma:inclusion subalg dual} that the subalgebras corresponding to these partial correct partitions are maximal. Indeed, the relations in (1) cannot be further refined while keeping $X \setminus \{ x \}$ saturated without getting the identity relation on the whole $X$; and the only relation finer than the relations in (2) is the total identity relation because their domain is already $X$. It is then sufficient to show that every proper subalgebra $\mathfrak B$ of $\upsets(X)$ is contained in a subalgebra corresponding to a partial correct partition of one of these two kinds.

Suppose $\mathfrak B$ corresponds to the partial correct partition $\sim$ of $X$ with domain $D \subseteq X$. If $D \neq X$, then we pick $x \in X \setminus D$. By Lemma~\ref{lemma:inclusion subalg dual}, the partial identity relation with domain $X \setminus \{x\}$ corresponds to a subalgebra containing $\mathfrak B$. If $D=X$ and $\sim$ is not the identity relation, take $x$ maximal among the elements of $X$ that have non-singleton equivalence classes with respect to $\sim$. Let $y \sim x$ and $y \ne x$. Since $x$ is maximal and $\sim$ is a total correct partition, $\coverof x=\coverof y$. Therefore, the equivalence relation $\sim'$ whose only non-trivial equivalence class is $\{x,y\}$ is a total correct partition, and by Lemma~\ref{lemma:inclusion subalg dual} it corresponds to a subalgebra containing $\mathfrak B$.

Since the partial correct partitions in (1) are the identity relations on their domains, the subalgebras corresponding to them are total, whereas the total partial correct partitions in (2) correspond to strict Heyting subalgebras
by Lemma~\ref{lem:total and strict hey dually}.
\end{proof}

\begin{theorem}\label{thm:dual maximal nuclear subalg}
Let $(X,S)$ be a finite S-poset. Maximal nuclear subalgebras of $(\upsets(X),j_S)$ correspond to partial correct partitions $\sim$ of $X$ with domain $D$ such that either
\begin{enumerate}
\item $D=X \setminus \{ x \}$ for some $x \in X$, $\sim$ is the identity relation on $D$, and $x \notin S$ or $\coverof x \subseteq S$, or
\item $D=X$, the only non-trivial equivalence class of $\sim$ is $\{x,y\}$ with $\coverof x=\coverof y$, and $x \in S$ iff $y \in S$.
\end{enumerate}
Maximal nuclear subalgebras of the first kind are total, and the ones of the second kind are strict Heyting.
Furthermore, maximal nuclear subalgebras are exactly the maximal subalgebras that are nuclear.
\end{theorem}

\begin{proof}
We first show that the correct partitions in (1) and (2) are nuclear. Let $\sim$ be a partial correct partition such that $D=X \setminus \{ x \}$ and $\sim$ is the identity on $D$. Then $\sim$ is nuclear iff $x \notin \max (S \cap \dn d)$ for all $d \in D$. This happens exactly when $x \notin S$ or $\coverof x \subseteq S$. Next let $\sim$ be a partial correct partition such that $D=X$ and the only non-trivial equivalence class of $\sim$ is $\{x,y\}$. Then $\sim$ is nuclear exactly when $\{x,y\}$ is either contained in or disjoint from $S$; that is, when $x \in S$ iff $y \in S$. Applying Proposition~\ref{prop:dual total and strict Hey nuclear subalg} then yields that the correct partitions in (1) and (2) are nuclear.

By Theorem~\ref{thm:dual maximal subalg}, the partial correct partitions in (1) and (2) correspond to maximal subalgebras. In particular, they are maximal nuclear subalgebras. It remains to show that every proper maximal nuclear subalgebra corresponds to a partial correct partition in (1) or (2). We do this by showing that each proper nuclear subalgebra is contained in a nuclear subalgebra corresponding to a partial correct partition of one of these two kinds.

Let $\sim$ be a nuclear partial correct partition with domain $D$. We assume that $\sim$ does not correspond to $\upsets(X)$; that is, $\sim$ is not the total identity partition. Let $Z$ be the union of $X \setminus D$ and the elements of $D$ whose equivalence class is nontrivial. Then $Z$ is nonempty. Note that if $z \in \max Z$, then $\up z \setminus \{ z \} \subseteq D$ and $\sim$ restricted to $\up z \setminus \{ z \}$ is the identity relation. Either $\max Z \subseteq D$ or $\max Z \nsubseteq D$.

First suppose that $\max Z \nsubseteq D$, and take $x \in (\max Z) \setminus D$. By Lemma~\ref{lemma:inclusion subalg dual}, the subalgebra corresponding to the identity relation with domain $X \setminus \{ x \}$ contains the subalgebra corresponding to $\sim$. It remains to show that the identity relation with domain $X \setminus \{ x \}$ is one of the relations in (1); that is, $x \notin S$ or $\coverof x \subseteq S$. Suppose $x \in S$ and let $d \in \coverof x$. Then $d \in D$ and since $\sim$ is nuclear, there are $s' \in D \cap S$, $d' \in D$ such that $x \le s' \le d'$ and $d \sim d'$. As $x \in \max Z$ and $x < d$, we have $d=d'$.
Since $x \notin D$ and $x \le s'$, we have $x < s'$.
Therefore, $x < s' \le d \in \coverof x$ implies $d=s'$. Thus, $\coverof x \subseteq S$.

Next suppose that $\max Z \subseteq D$. Let $x \in \max Z$ and $y$ be such that $x \neq y$ and $x \sim y$. We show that $y \in \max Z$. There is $z \in \max Z$ such that $y \le z$. Then $z \in D$ and since $\sim$ is a partial correct partition, there is $w \in D$ such that $x \le w$ and $w \sim z$.
We have that $w \in Z$ because either its equivalence class is nontrivial or else $w=z \in Z$. As $w \in Z$, $x\in \max Z$, and $x\le w$, we have $x=w$.
Therefore, $y \sim z$ and $y \le z$. Strictness then implies $y=z$. Thus, $y\in \max Z$. This yields that $\up x \setminus \{ x \}$ and $\up y \setminus \{ y \}$ are contained in $D$ and $\sim$ restricted to these sets is the identity relation. Since $\sim$ is a partial correct partition and $x \sim y$, we conclude that $\up x \setminus \{ x \}=\up y \setminus \{ y \}$, so $\coverof x = \coverof y$. Also, as $\sim$ is nuclear, $S \cap D$ is saturated, which implies $x \in S$ iff $y \in S$. Thus, by Lemma~\ref{lemma:inclusion subalg dual}, the total correct partition with the only nontrivial equivalence class $\{x,y\}$ corresponds to a nuclear subalgebra containing the nuclear subalgebra corresponding to $\sim$.

Finally, it follows from Theorem~\ref{thm:dual maximal subalg} that
maximal nuclear subalgebras of the first kind are total, the ones of the second kind are strict Heyting, and that
every maximal nuclear subalgebra is a maximal subalgebra.
\end{proof}

\section{Coloring technique, universal models, and local finiteness of $\nis$}\label{sec:coloring}

In this section we prove our main result that $\nis$ is locally finite, thus generalizing Diego's theorem. For this we adopt the well-known coloring technique \cite{EG77}, which we then use to construct $n$-universal models for each $n$. Our construction is an adaptation to our setting of similar constructions in modal and intuitionistic logics (see, e.g., \cite{Ghi95,CZ97,Nic06}  and the references therein). For each $n$, the construction builds the $n$-universal model recursively, layer by layer. While it follows from the construction that each layer is finite, it is not obvious at all that the construction eventually terminates. Indeed, it often does not in similar situations in modal and intuitionistic logics. One of our main observations is that the construction indeed terminates. From this we derive that $\nis$ is locally finite.

Let $\mathfrak A=(A,j)$ be a finite nuclear implicative semilattice and let $a_1,\dots,a_n\in A$. We start by adapting the coloring technique of \cite{EG77} which will allow us to determine whether $a_1,\dots,a_n$ generate $\mathfrak A$. We would like to stress that while the coloring technique of \cite{EG77} applies to an arbitrary Heyting algebra, we will only be concerned with finite nuclear implicative semilattices. The reason being that if $\mathfrak A$ is finite, then we may use the results of Section~\ref{sec: finite duality nuclear implicative semilattices} and assume without loss of generality that $\mathfrak A$ is the algebra $(\upsets(X),j_S)$ of upsets of some finite S-poset $(X,S)$. The elements $a_1,\dots,a_n$ are then upsets of $X$, and give rise to ``coloring" of $X$.

\begin{convention}
From now on $n \ge 1$ will be fixed. Colors will be subsets of $\{ 1, \ldots, n \}$. If $n=0$, then we assume that $\{1, \ldots, n \}=\emptyset$, so $\emptyset$ is the only available color.
\end{convention}

The next definition works for an arbitrary (not necessarily finite) S-poset.

\begin{definition}
\begin{enumerate}
\item[]
\item A \textit{coloring} of an S-poset $(X,S)$ is a function $\colorsof: X \to \powerset(\{1, \ldots, n \})$ such that $x \le y$ implies $\colorsof(x) \subseteq \colorsof(y)$.
\item A \textit{model} is a triple $\mathfrak M=(X,S,c)$ where $(X,S)$ is an S-poset and $c$ is a coloring of $(X,S)$.
\end{enumerate}
\end{definition}

For $Y \subseteq X$ we let
\begin{equation*}\label{color of a subset} \tag{$\ddagger$}
\colorsof(Y)= \bigcap\{ \colorsof(x) \mid x \in Y \}.
\end{equation*}
In particular, $\colorsof(\emptyset)=\{1, \ldots, n \}$. We think of $\colorsof$ as a function associating to each element of $X$ one of $2^n$ colors. We refer to $\colorsof(x)$ as the \emph{color of} $x$,
and to $c(Y)$ as the \textit{color of} $Y$.

\begin{remark}\label{rem:colorings}
There is a one-to-one correspondence between colorings of $(X,S)$ and $n$-tuples $U_1, \ldots, U_n$ of upsets of $X$. Indeed, each $n$-tuple $U_1, \ldots, U_n \in \upsets(X)$ gives rise to the coloring $\colorsof: X \to \powerset(\{1, \ldots, n \})$ given by
\[
\colorsof(x)=\{i \in \{1, \ldots, n \} \mid x \in U_i \}.
\]
Conversely, each coloring gives rise to the $n$-tuple $U_1, \ldots, U_n \in \upsets(X)$ given by
\[
U_i=\{ x \in X \mid i \in \colorsof(x) \}
\]
for each $i=1, \ldots, n$.
\end{remark}

\begin{definition}
We say that a finite model $\mathfrak M=(X,S, \colorsof)$ is \textit{irreducible} if the nuclear implicative semilattice $(\upsets(X),j_S)$ is generated by the upsets $U_1, \ldots, U_n$.
\end{definition}

\begin{remark}\label{remark:finite n-gen and irred models}
It follows from the duality for finite nuclear implicative semilattices and Remark~\ref{rem:colorings} that there is a one-to-one correspondence between finite irreducible models and finite $n$-generated nuclear implicative semilattices, which is obtained by associating with each finite irreducible model $\mathfrak M=(X,S, \colorsof)$ the finite nuclear implicative semilattice $(\upsets(X),j_S)$ generated by $U_1,\dots,U_n$ where $U_i=\{ x \in X \mid i \in \colorsof(x) \}$ for each $i$.

In addition, homomorphisms of finite nuclear implicative semilattices that send generators to generators correspond to total one-to-one S-morphisms that preserve the coloring.
To define what it means, let $\mathfrak M=(X,S, \colorsof)$ and $\mathfrak N=(Y,T, \colorsof)$ be two models. We say that an S-morphism $f:\mathfrak{M} \to \mathfrak{N}$ \emph{preserves the coloring} if  $\colorsof (f(x))=\colorsof(x)$ for all $x \in X$. By Proposition~\ref{prop:one-to-one and onto dual correspondence}(2), $f:\mathfrak{M} \to \mathfrak{N}$ is dual to an onto homomorphism iff $f$ is total and one-to-one.
This is equivalent to $i \in \colorsof( f(x))$ iff $i \in \colorsof (x)$ for all $x \in X$ and $i \in \{1, \ldots, n \}$, which is the same as requiring that $\{x \in X \mid i \in \colorsof(x)\}=f^{-1}(\{y \in Y \mid i \in \colorsof(y)\})$ for all $i \in \{1, \ldots, n \}$. When $f$ is one-to-one and total, $f^*=f^{-1}$. Thus, $f$ preserves the coloring iff $f^*$ maps generators to generators.
\end{remark}

\begin{lemma}\label{lem:elements of dual of partition}
Let $X$ be a finite poset, $\mathfrak B$ a subalgebra of $\upsets(X)$, and $\sim$ the corresponding partial correct partition of $X$ with domain $D$. For an upset $U$ of $X$ we have that $U \in \mathfrak B$ iff $U \cap D$ is saturated and $\max(X \setminus U) \subseteq D$.
\end{lemma}

\begin{proof}
Let $Y:=X/ {\sim}$ be the quotient and let $\pi: X \to Y$ be the \kmor{} sending $x\in D$ to its equivalence class $[x]$.
By Proposition~\ref{prop:subalgebras of U(X) and correct partitions},
$U \in \mathfrak B$ iff $U=\pi^*(V)$ for some upset $V$ of $Y$.  Therefore,
\begin{align*}
U \in \mathfrak B & \mbox{ iff } U=\pi^*(V) \mbox{ for some } V \in \upsets(Y) \\
& \mbox{ iff } U=X \setminus \dn \pi^{-1}(Y \setminus V)\mbox{ for some } V \in \upsets(Y) \\
& \mbox{ iff } U= \{ x \in X \mid \pi(\up x) \subseteq V \}\mbox{ for some } V \in \upsets(Y).
\end{align*}

First suppose that $U \in \mathfrak B$. Then $U= \{ x \in X \mid \pi(\up x) \subseteq V \}$ for some $V \in \upsets(Y)$. Let $x \in U \cap D$
and $x \sim y$, so $\pi(x)=\pi(y)$.
By Remark~\ref{rem:p-morphism}(2),
\[
\pi(\up y)=\up \pi(y)=\up \pi(x)=\pi(\up x) \subseteq V.
\]
Therefore, $y \in U$, and so $U \cap D$ is saturated. In addition, $x \in X \setminus U$ iff $\pi(\up x) \nsubseteq V$, which happens iff there is $y\in\up x\cap D$ with $\pi(y)\notin V$. Thus, if $x \in \max(X \setminus U)$, then such a $y$ has to be $x$, yielding that $x \in D$. Consequently, $\max(X \setminus U) \subseteq D$.

Conversely, suppose that $U \cap D$ is saturated and $\max(X \setminus U) \subseteq D$.
Since $U \in \upsets(X)$, we have $\pi(U) \in \upsets(Y)$ by Remark~\ref{rem:p-morphism}(2).
Let $V=\pi(U)$. We show $U= \{ x \in X \mid \pi(\up x) \subseteq V \}$, which yields that $U \in \mathfrak B$. Let $x \in U$. Since $U \in \upsets(X)$, we have $\up x \subseteq U$, so $\pi(\up x) \subseteq \pi(U)=V$. Let $x \notin U$. Then there is $z \in \up x \cap \max(X \setminus U)$. By our assumption, $z\in D$, so $\pi(z) \in \pi(\up x)$. On the other hand, $\pi(z) \notin \pi(U)$ because $U \cap D$ is saturated and $z \in D \setminus U$. Thus, $\pi(\up x) \nsubseteq \pi(U)=V$.
\end{proof}

\begin{lemma}\label{lem:U_i in dual of partition}
Let $\mathfrak M=(X,S, \colorsof)$ be a finite model and $\sim$ a partial correct partition of $X$ with domain $D$. The upsets $U_1, \ldots, U_n$ of $X$ belong to the subalgebra of $\upsets(X)$ corresponding to $\sim$ iff the following two conditions are satisfied:
\begin{enumerate}
\item if $x \sim y$, then $\colorsof(x)=\colorsof(y)$,
\item if $\colorsof(x) \neq \colorsof(\coverof x)$, then $x\in D$.
\end{enumerate}
\end{lemma}

\begin{proof}
Let $\mathfrak B$ be the subalgebra of $(\upsets(X),j_S)$ corresponding to $\sim$. By Lemma~\ref{lem:elements of dual of partition} we have that an upset $U$ of $X$ is in $\mathfrak B$ iff $U \cap D$ is saturated and $\max(X \setminus U) \subseteq D$. We translate these two conditions for the upsets $U_i$ in terms of the coloring $\colorsof$. Recall that $U_i=\{ x \in X \mid i \in \colorsof(x) \}$ for each index $i$. Therefore, $U_i \cap D$ is saturated iff $x \sim y$ implies that $i \in \colorsof(x) \Leftrightarrow i \in \colorsof(y)$. Thus, $U_i \cap D$ is saturated for all $i$ iff $x \sim y$ implies $\colorsof(x)=\colorsof(y)$. For the second condition, since $X \setminus U_i = \{x \in X \mid i \notin \colorsof(x) \}$, we have $\max(X \setminus U_i)= \{x \in X \mid i \notin \colorsof(x), \; i \in \colorsof(\coverof x) \}$. Since $\colorsof(x) \subseteq \colorsof(\coverof x)$ for all $x \in X$, we have $\bigcup_i \max(X \setminus U_i)=\{x \in X \mid \colorsof(x) \neq \colorsof(\coverof x)\}$. Consequently, $\max(X \setminus U_i) \subseteq D$ for all $i$ iff $\bigcup_i \max(X \setminus U_i) \subseteq D$ iff $\{x \in X \mid \colorsof(x) \neq \colorsof(\coverof x)\} \subseteq D$.
\end{proof}

\begin{theorem} [Coloring Theorem]\label{thm:conditions for irred model}
A finite model $\mathfrak M=(X, S,\colorsof)$ is irreducible iff the following two conditions are satisfied:
\begin{enumerate}
\item $\forall x\in X \left(\colorsof(x)=\colorsof(\coverof x)\Rightarrow x\in S\ \&\ {\coverof x}\nsubseteq S\right)$,
\item $\forall x,y\in X \left(\coverof x=\coverof y\ \&\ \colorsof(x)=\colorsof(y)\ \&\ (x\in S\Leftrightarrow y\in S)\Rightarrow x=y\right)$.
\end{enumerate}
\end{theorem}

\begin{proof}
By definition, $\mathfrak M$ is irreducible iff $(\upsets(X),j_S)$ is generated by $U_1, \ldots, U_n$. Therefore, $\mathfrak M$ is irreducible iff there is no proper nuclear subalgebra of $(\upsets(X),j_S)$ containing $U_1, \ldots, U_n$. Since every proper nuclear subalgebra is contained in a maximal nuclear subalgebra, $\mathfrak M$ is irreducible iff there is no maximal nuclear subalgebra of $(\upsets(X),j_S)$ containing $U_1, \ldots, U_n$. By Theorem~\ref{thm:dual maximal nuclear subalg}, maximal nuclear subalgebras of $(\upsets(X),j_S)$ correspond to partial correct partitions $\sim$ of $X$ with domain $D$ such that either
\begin{enumerate}
\item $D=X \setminus \{ x \}$ for some $x \in X$, $\sim$ is the identity relation on $D$, and $x \notin S$ or $\coverof x \subseteq S$, or
\item $D=X$, the only non-trivial equivalence class of $\sim$ is $\{x,y\}$ with $\coverof x=\coverof y$, and $x \in S$ iff $y \in S$.
\end{enumerate}
By Lemma~\ref{lem:U_i in dual of partition}, all the $U_i$ are in the subalgebra corresponding to a partition of the first kind iff $\colorsof(x)=\colorsof(\coverof x)$. So excluding the existence of any such partition containing all the $U_i$ is equivalent to requiring that there is no $x \in X$ such that $\colorsof(x)=\colorsof(\coverof x)$, and $x \notin S$ or $\coverof x \subseteq S$ at the same time. This is exactly the condition
\[
\forall x\in X \left(\colorsof(x)=\colorsof(\coverof x)\Rightarrow x\in S\ \&\ {\coverof x}\nsubseteq S\right).
\]
By Lemma~\ref{lem:U_i in dual of partition}, all the $U_i$ are in the subalgebra corresponding to a partition of the second kind iff $\colorsof(x)=\colorsof(y)$. So excluding the existence of any such partition containing all the $U_i$ is equivalent to requiring that there are no elements $x \neq y$ such that $x \sim y$, $\coverof x=\coverof y$, $x \in S$ iff $y \in S$, and $\colorsof(x)=\colorsof(y)$ at the same time. This is exactly the condition
\[
\forall x,y\in X \left(\coverof x=\coverof y\ \&\ \colorsof(x)=\colorsof(y)\ \&\ (x\in S\Leftrightarrow y\in S)\Rightarrow x=y\right).
\]
\end{proof}

\begin{remark}\label{rem:full color not allowed}
Let $\mathfrak M=(X, S,\colorsof)$ be a finite irreducible model.
\begin{enumerate}
\item Condition~(1) of Theorem~\ref{thm:conditions for irred model} implies that no element of $X$ has full color $\{1, \ldots, n\}$. Indeed, since $X$ is finite, every element is below some maximal element. Therefore, if some element had full color, there would exist $x \in \max X$ such that $\colorsof (x)=\{1, \ldots, n\}$. Thus, $\colorsof(x)=\{1, \ldots, n\}= \colorsof(\emptyset)=\colorsof(\coverof x)$. By Condition~(1), $\emptyset=\coverof x \nsubseteq S$, a contradiction.
Since $U_1\cap\cdots\cap U_n=\{ x \in X \mid \colorsof(x)=\{1, \ldots, n\} \}$, we conclude that $U_1\cap\cdots\cap U_n=\varnothing$.
Dually this means that if a finite nuclear implicative semilattice $\mathfrak A$ is generated by $g_1, \ldots, g_n$, then $g_1 \meet \cdots \meet g_n$ has to be the bottom of $\mathfrak A$.
\item Condition~(2) of Theorem~\ref{thm:conditions for irred model} implies that if $x,y \in \max X$ are distinct and $x \in S$ iff $y \in S$, then $\colorsof(x)\ne\colorsof(y)$.
\end{enumerate}
\end{remark}

We next utilize the Coloring Theorem to construct $n$-universal models. Let $n$ be a fixed nonnegative integer. We assume that the coloring maps of all the models we consider are into $\powerset(\{1, \ldots, n\})$.

\begin{definition}
A model $\mathfrak L=(X,S,\colorsof)$ is \emph{$n$-universal} provided for every finite irreducible model $\mathfrak M=(Y,T,\colorsof)$ there is a unique embedding of posets $e:Y \to X$ such that $e(Y)$ is an upset of $X$, $e^{-1}(S)=T$, and $\colorsof (e(y))=\colorsof(y)$ for all $y \in Y$.
\end{definition}

\begin{definition}
The \textit{height} of a poset $(X,\le)$ is the supremum of the cardinalities of finite chains in $X$. The \textit{height of $x \in X$} is the height of the poset $\up x$. The \textit{height} of a model is the height of the underlying poset.
\end{definition}

We construct the $n$-universal model $\mathfrak L$ recursively, building it layer by layer, by constructing a sequence of finite irreducible models
\[
\mathfrak L_0\subseteq \mathfrak L_1\subseteq\cdots\subseteq \mathfrak L_k\subseteq\cdots
\]
Each $\mathfrak L_k$ in the sequence has height
$k$. The $n$-universal model $\mathfrak L$ is then the union of the models $\mathfrak L_k$.
Below we use $\subset$ to denote proper inclusion.

\begin{definition}\label{def:universal model}
For each $k \geq 0$, define the model $\mathfrak L_k=(X_k, S_k,\colorsof_k)$ recursively as follows.

{\bf Base case:}
Define $\mathfrak L_0 = (X_0,S_0,\colorsof_0)$ by setting $X_0,S_0=\varnothing$ and $\colorsof_0$ to be the empty map.

For $\sigma \subseteq \{ 1, \ldots, n \}$ consider the formal symbols $r_{\emptyset, \sigma}$ and $s_{\emptyset, \sigma}$. Then define $\mathfrak L_1 = (X_1,S_1,\colorsof_1)$ by setting
\begin{itemize}
\item $X_1=\{r_{\varnothing,\sigma}, s_{\varnothing,\sigma} \mid \sigma \subset \{ 1, \ldots, n\}\}$ and $\le_1$ is the identity relation on $X_1$,
\item $S_1=\{s_{\varnothing,\sigma} \mid \sigma \subset \{ 1, \ldots, n\} \}$,
\item $\colorsof_1(r_{\varnothing,\sigma}) = \colorsof_1(s_{\varnothing,\sigma} ) = \sigma$.
\end{itemize}

{\bf Recursive step:} Suppose $\mathfrak L_k=(X_k,S_k,\colorsof_k)$ is already constructed for $k\ge 1$. For $\alpha \subseteq X_k$ and $\sigma \subseteq \{ 1, \ldots, n \}$ consider the formal symbols $r_{\alpha, \sigma}$ and $s_{\alpha, \sigma}$. Then define $\mathfrak L_{k+1}=(X_{k+1},S_{k+1},\colorsof_{k+1})$ by setting
\begin{itemize}
\item $X_{k+1}$ is obtained by adding for each antichain $\alpha \subseteq X_k$ with $\alpha \nsubseteq X_{k-1}$ the following new elements to $X_k$:
\begin{enumerate}
\item $r_{\alpha, \sigma}$ for each $\sigma \subset \colorsof_k(\alpha)$,
\item $s_{\alpha, \sigma}$ for each $\sigma \subset \colorsof_k(\alpha)$,
\item $s_{\alpha, \colorsof_k(\alpha)}$ if $\alpha \nsubseteq S_k$.
\end{enumerate}
The partial order on $X_{k+1}$ extends the partial order on $X_k$ so that the covers of the elements of $X_{k+1}\setminus X_k$ are defined as $\coverof r_{\alpha, \sigma}=\coverof s_{\alpha, \sigma}=\alpha$.
\item $S_{k+1}$ is obtained by adding to $S_k$ the elements of $X_{k+1} \setminus X_k$ of the form $s_{\alpha, \sigma},s_{\alpha, \colorsof_k(\alpha)}$.
\item $\colorsof_{k+1}$ extends $\colorsof_k$ so that $\colorsof_{k+1}(r_{\alpha, \sigma})=\colorsof_{k+1}(s_{\alpha, \sigma})=\sigma$ and $\colorsof_{k+1}(s_{\alpha, \colorsof_k(\alpha)})=\colorsof_k(\alpha)$.
\end{itemize}
Finally, we define $\mathfrak L=(X,S,\colorsof)$ by setting
\[
X = \bigcup_k X_k, \ S = \bigcup_k S_k, \mbox{ and } \colorsof(x) = c_k(x) \mbox{ if } x \in X_k.
\]
\end{definition}

\begin{remark}\label{rem:layers}
\begin{enumerate}
\item[]
\item It follows from the construction that each $\mathfrak L_k$ is finite. Therefore, each $\mathfrak L_k$ is an irreducible model by the Coloring Theorem.
\item No element of $\mathfrak L$ has full color.
\item Each nonempty layer increases the height of the model by $1$. Therefore, if the $k$-th layer is nonempty, then the height of $\mathfrak L_k$ is $k$.
In fact,
$\mathfrak{L}_k$ is the set of elements of $\mathfrak{L}$ whose height is $\le k$.
\item Rules (1) and (2) decrease the color of the new elements added. However, Rule (3) does not. Because of this, it is unclear whether the construction terminates.
For example, all $n$-universal models for Heyting algebras with $n>0$ are infinite because it is easy to add elements at each layer without making their color decrease (see, e.g., \cite[Sec.~3.2]{Nic06}). We will address the issue of termination in Theorem~\ref{thm:universal model is finite}.
\end{enumerate}
\end{remark}

\begin{theorem} \label{thm:universal model}
The model $\mathfrak L=(X,S,\colorsof)$ is $n$-universal.
\end{theorem}

\begin{proof}
Let $\mathfrak M=(Y,T,\colorsof)$ be a finite irreducible model.
We prove by induction on the height of $Y$ that there is a unique embedding $e:Y \to X$ such that $e(Y)$ is an upset of $X$, $e^{-1}(S)=T$, and $\colorsof (e(y))=\colorsof(y)$ for all $y \in Y$.
If $Y$ is empty, there is nothing to prove.

If the height of $Y$ is $1$, then the partial order on $Y$ is the identity, and we define
\begin{align*}
e(y)=
\begin{cases}
r_{\emptyset,\colorsof(y)} \quad & \mbox{if }y \notin T, \\
s_{\emptyset,\colorsof(y)} \quad & \mbox{if }y \in T.
\end{cases}
\end{align*}
It is straightforward to see that $e(Y)$ is an upset of $X$, that $e^{-1}(S)=T$, and that $\colorsof (e(y))=\colorsof(y)$ for all $y \in Y$.
Thus, the embedding must be unique by Remark~\ref{rem:full color not allowed}(2).

If the height of $Y$ is $m+1$, let $Y'$ be the set of elements of $Y$ of height less than or equal to $m$ and let $T'=T \cap Y'$.
Since $\mathfrak M$ is an irreducible model, so is $\mathfrak M'=(Y',T',\colorsof_{|Y'})$.
Therefore, by the inductive hypothesis, there is a unique embedding $e':Y' \to X$ such that $e'(Y')$ is an upset of $X$, $(e')^{-1}(S)=T'$ and $\colorsof(e'(y))=\colorsof(y)$ for all $y \in Y'$. For each $y \in Y \setminus Y'$ we have $\coverof y \subseteq Y'$, so we can define $e:Y \to X$ by extending $e'$ as follows:
\begin{align*}
e(y)=
\begin{cases}
r_{e'(\coverof y),\colorsof(y)} \quad & \mbox{if }y \notin T, \\
s_{e'(\coverof y),\colorsof(y)} \quad & \mbox{if }y \in T.
\end{cases}
\end{align*}
To see that $e$ is well defined,
since $e'$ embeds $Y'$ into $X$, we have that
$e'(\coverof y)$ is an antichain in $X_m$ which is not entirely contained in $X_{m-1}$.
As $\mathfrak M$ is an irreducible model, by Theorem~\ref{thm:conditions for irred model}, if $y \notin T$, then $\colorsof(y) \subset \colorsof(\coverof y)=\colorsof(e'(\coverof y))$.
So Rule (1) applies, and hence $r_{e'(\coverof y),\colorsof(y)}$ exists in $X$.
Suppose $y \in T$. We have $\colorsof(y) \subset\colorsof(\coverof y)$ or $\colorsof(y)=\colorsof(\coverof y)$. In the former case, $s_{e'(\coverof y),\colorsof(y)}$ exists in $X$
by Rule (2).
In the latter case, since $\mathfrak M$ is an irreducible model, Theorem~\ref{thm:conditions for irred model} gives $\coverof y \nsubseteq T'$. Therefore, $e'(\coverof y) \nsubseteq S$, so $ s_{e'(\coverof y),\colorsof(y)}$ exists in $X$
by Rule (3).
Thus, $e$ is well defined.

It follows from the construction of $e$ and $\mathfrak{L}$ that $y \in \coverof x $ iff $e(y) \in \coverof e(x)$. Therefore, an easy induction shows that for all $x,y \in Y$ we have $x \le y$ iff $e(x) \le e(y)$. Thus, $e$ is an embedding.
Moreover, the definition of $e$ implies that $e(Y)$ is an upset, that $e^{-1}(S)=T$, and that $\colorsof(e(y))=e(y)$. Furthermore,
we are forced to extend $e'$ in this way if we want $e^{-1}(S)=T$ and $\colorsof (e(y))=\colorsof(y)$ for all $y \in Y$. Thus, $e$ is unique.
\end{proof}

As we pointed out in Remark~\ref{rem:layers}(1), each layer of $\mathfrak L$ is finite. We next show a lot stronger result, that the construction of $\mathfrak L$ terminates, and hence that $\mathfrak L$ is finite. For this we introduce the following notation.

For $d\le n$, let
\[
X^d= \{ x \in X \mid |\colorsof(x)|=d \} \mbox{ and } X^{\geq d}= \{ x \in X \mid |\colorsof(x)| \geq d \}.
\]
We also let $R=X \setminus S$, and define $S^d$, $S^{\geq d}$, $R^d$, and $R^{\geq d}$ similarly.

In addition, let
\[
S^d_= =\{ s_{\alpha, \colorsof(\alpha)} \mid |\colorsof(\alpha)|=d \} \mbox{ and } S^d_<=\{ s_{\alpha, \sigma}  \mid \sigma \subset \colorsof(\alpha)
, \;|\sigma|=d
 \}.
\]
We then have $S^d=S^d_= \cup S^d_<$.

Observe that since $X=R \cup S$, we have
\[
X^{\geq d}=X^{\geq d+1} \cup X^d=X^{\geq d+1} \cup R^d \cup S_=^d \cup S_<^d.
\]
Also observe that
\[
X^{\geq n}=X^n=\{x \in X \mid \colorsof(x)=\{1, \ldots, n \} \}=\varnothing.
\]

\begin{lemma}\label{lem:univ model finiteness}
Let $\mathfrak{L}=(X,S,\colorsof)$ be the $n$-universal model and $d<n$.
\begin{enumerate}
\item There is a one-to-one map $R^d \to \powerset(X^{\geq d+1}) \times \powerset(\{1, \ldots, n \})$. Therefore, if $X^{\geq d+1}$ is finite, then so is $R^d$.
\item There is a one-to-one map $S^d_< \to \powerset(X^{\geq d+1}) \times \powerset(\{1, \ldots, n \})$. Therefore, if $X^{\geq d+1}$ is finite, then so is $S^d_<$.
\item If $C \subseteq S^d_=$ is a chain, then there is a one-to-one map $C \to R^{\geq d}$. Therefore, if $R^{\geq d}$ is finite, then $S^d_=$ has finite height.
\end{enumerate}
\end{lemma}

\begin{proof}
(1) Define a map $R^d \to \powerset(X^{\geq d+1}) \times \powerset(\{1, \ldots, n \})$ by sending $x \in R^d$ to $(\coverof x, \colorsof(x))$. Since $x \in R^d$,
by Rule (1),
we have $x=r_{\coverof x, \colorsof(x)}$ with $\coverof x \subseteq X^{\geq d+1}$. Therefore, the map is one-to-one. Since $\powerset(\{1, \ldots, n \})$ is finite, if $X^{\geq d+1}$ is finite, then so is $R^d$.

(2) Similarly to (1), define a one-to-one map $S^d_< \to \powerset(X^{\geq d+1}) \times \powerset(\{1, \ldots, n \})$ by sending $x \in S^d_<$ to $(\coverof x, \colorsof(x))$.
By Rule (2),
$x=s_{\coverof x, \colorsof(x)}$ with $\coverof x \subseteq X^{\geq d+1}$. Therefore, the map is one-to-one. Thus, if $X^{\geq d+1}$ is finite, then so is $S^d_<$.

(3) Let $C$ be a chain in $S^d_=$.
By Rule (3),
for each $x \in C$ there is $r \in \coverof x \cap R^{\geq d}$. Choosing $r_x \in \coverof x \cap R^{\geq d}$ for each $x \in C$ and sending $x$ to $r_x$ defines a map $C \to R^{\geq d}$. This map is one-to-one because if $x, y \in C$ with $x \neq y$, then $x<y$ or $y<x$; and in either case, $\coverof x \cap \coverof y = \emptyset$. Thus, if $R^{\geq d}$ is finite, every chain in $S^d_=$ has to be finite with cardinality at most $|R^{\geq d}|$. In particular, if $R^{\geq d}$ is finite, then $S^d_=$ has finite height.
\end{proof}

\begin{theorem}\label{thm:universal model is finite}
The $n$-universal model $\mathfrak{L}$ is finite.
\end{theorem}

\begin{proof}
We show by induction that $X^{\geq d}$ is finite for each $0 \le d \le n$. This will imply that $X=X^{\geq 0}$ is finite. We proceed by reverse induction, decreasing $d$ at each step starting from $d=n$. For the base case, we already observed that $X^{\geq n}=\emptyset$.

For the inductive step, let $d < n$ and $X^{\geq d+1}$ be finite. We first show that $X^{\geq d}$ has finite height. Since $X^{\geq d}=X^{\geq d+1} \cup R^d \cup S_=^d \cup S_<^d$, it is sufficient to observe that $X^{\geq d+1}$, $R^d$, $S_=^d$, and $S_<^d$ have finite height. By inductive hypothesis, $X^{\ge d+1}$ is finite. Since $X^{\ge d+1}$ is finite, Lemma~\ref{lem:univ model finiteness} implies that $R^d$ and $S^d_<$ are finite. Therefore, $X^{\ge d+1}$, $R^d$, and $S^d_<$ have finite height. Since $R^{\geq d} \subseteq R^d \cup X^{\geq d+1}$ and both $R^d$, $X^{\geq d+1}$ are finite, so is $R^{\geq d}$. By Lemma~\ref{lem:univ model finiteness}(3), $S^d_=$ has finite height. Thus, $X^{\geq d}$ has finite height, say $m$.
Remark~\ref{rem:layers}(3) then implies that $X^{\geq d} \subseteq X_m$. Since $X_m$ is finite by Remark~\ref{rem:layers}(1), $X^{\geq d}$ is finite.
\end{proof}

\begin{remark}\label{rem:univ model is irreducible}
Since the $n$-universal model $\mathfrak{L}$ is finite, it coincides with $\mathfrak{L}_k$ for some $k$. Therefore, $\mathfrak{L}$ is irreducible by Remark~\ref{rem:layers}(1).
\end{remark}

We are ready to prove our main result.

\begin{theorem}\label{thm:nis locally finite}
$\nis$ is locally finite.
\end{theorem}

\begin{proof}
For each $n$ let $\mathfrak F_n$ be the free $n$-generated nuclear implicative semilattice. It is sufficient to prove that $\mathfrak F_n$ is finite. Let $\{\mathfrak A_\alpha\}$ be the inverse system of finite homomorphic images of $\mathfrak F_n$. Then each $\mathfrak A_\alpha$ is $n$-generated. The
bonding
maps of this inverse system are homomorphisms mapping generators to generators. Let $\mathfrak M_\alpha$ be the finite irreducible model corresponding to $\mathfrak A_\alpha$. Then $\{\mathfrak M_\alpha\}$ is a direct system of finite irreducible models. By Remark~\ref{remark:finite n-gen and irred models}, the maps of this direct system are S-morphisms preserving the coloring. By Theorem~\ref{thm:universal model is finite}, the $n$-universal model $\mathfrak{L}$ is finite. Therefore, $\mathfrak{L}$ is the terminal element of $\{\mathfrak M_\alpha\}$. By~\cite[Ex.~11.4.5]{AHS06}, the direct limit of $\{\mathfrak M_\alpha\}$ is isomorphic to $\mathfrak{L}$. Thus, the inverse limit of $\{\mathfrak A_\alpha\}$ is isomorphic to $\mathfrak{L}^*$. By Theorem~\ref{thm:nis generated by finite algebras}, $\nis$ is generated by its finite algebras. Consequently, $\mathfrak{F}_n$ embeds into the inverse limit of $\{\mathfrak A_\alpha\}$ (see, e.g., \cite[Prop.~2.1]{BM09}). Therefore, since $\mathfrak{L}^*$ is finite, so must be $\mathfrak{F}_n$. Thus, $\nis$ is locally finite.
\end{proof}

\begin{remark}
As follows from the above proof, $\mathfrak{F}_n$ embeds into $\mathfrak{L}^*$. In fact, $\mathfrak{F}_n$ is isomorphic to $\mathfrak{L}^*$. Indeed, since $\mathfrak{L}$ is finite, it is irreducible
by Remark~\ref{rem:univ model is irreducible}.
Therefore, $\mathfrak{L}^*$ is $n$-generated, and so it is a finite quotient of $\mathfrak{F}_n$.
Thus, $|\mathfrak{L}^*| \le |\mathfrak{F}_n|$. On the other hand, since $\mathfrak{F}_n$ embeds into $\mathfrak{L}^*$, we have $|\mathfrak{F}_n| \le |\mathfrak{L}^*|$. Consequently, the embedding of $\mathfrak{F}_n$ into $\mathfrak{L}^*$ is also onto, hence an isomorphism.
The isomorphism maps the free generators of $\mathfrak{F}_n$ to the upsets $U_1, \ldots, U_n$ of $\mathfrak L$ defined by the coloring $\colorsof$.
\end{remark}

\section{Bounded case}\label{sec:bounded case}

As we pointed out in Remark~\ref{rem:bounds}(1), every implicative semilattice has a top, but may not have a bottom. We call an implicative semilattice \emph{bounded} if it has a bottom, and an implicative semilattice homomorphism \emph{bounded} if it preserves the bottom. Let $\isbot$ be the category of bounded implicative semilattices and bounded implicative semilattice homomorphisms. Diego's theorem remains true for $\isbot$, and so $\isbot$ is locally finite. In this section we show that Theorem~\ref{thm:nis locally finite} also remains true for the category $\nisbot$ of bounded nuclear implicative semilattices and bounded nuclear
homomorphisms.

Clearly each finite implicative semilattice is bounded, so it has a bottom element which we denote by $0$. However, implicative semilattice homomorphisms between finite implicative semilattices do not have to preserve $0$. The next proposition gives a dual characterization of when they do. For this we recall that a subset $Y$ of a poset $(X,\le)$ is \emph{cofinal} if $\dn Y=X$. If $X$ is finite, then it is obvious that $Y$ is cofinal iff $\max X\subseteq Y$.

\begin{proposition}\label{prop:cofinal domain dual to preserving 0}
Let $f:X \to Y$ be a \kmor{} between finite posets with domain $D \subseteq X$. The implicative semilattice homomorphism $f^*:Y^* \to X^*$ is bounded iff $D$ is cofinal in $X$.
\end{proposition}

\begin{proof}
The bottom element of $\upsets(X)$ is $\emptyset$. By the definition of $f^*$, we have
\[
f^*(\emptyset)=X \setminus \dn f^{-1}(Y \setminus \emptyset )=X \setminus \dn D.
\]
Therefore, $f^*(\emptyset)=\emptyset$ iff $X \setminus \dn D=\emptyset$ iff $\dn D = X$.
\end{proof}

\begin{definition}
Let $\nisfbot$ be the full subcategory of $\nisbot$ consisting of finite nuclear implicative semilattices. Let also $\spfbot$ be the category of finite S-posets and \smor{}s with cofinal domain.
\end{definition}

Let $(\;)^*:\spf \to \nisf$ and $(\;)_*:\nisf \to \spf$ be the functors defined in Section~\ref{sec: finite duality nuclear implicative semilattices}. As a consequence of Theorem~\ref{thm:duality for nisf} and Proposition~\ref{prop:cofinal domain dual to preserving 0} we obtain that restricting these functors to $\spfbot$ and $\nisfbot$ yields the following dual equivalence.

\begin{theorem}
$\nisfbot$ is dually equivalent to $\spfbot$.
\end{theorem}

Let $A$ be a bounded implicative semilattice. We call a subalgebra $B$ of $A$ \emph{bounded} if $B$ contains the bottom element of $A$.

\begin{theorem}\label{thm:dual maximal bounded nuclear subalg}
Let $(X,S)$ be a finite S-poset. Maximal bounded nuclear subalgebras of $(\upsets(X), j_S)$ correspond to partial correct partitions $\sim$ of $X$ with domain $D$ such that either
\begin{enumerate}
\item $D=X \setminus \{ x \}$ for some $x \in X \setminus \max X$, $\sim$ is the identity relation on $D$, and $x \notin S$ or $\coverof x \subseteq S$, or
\item $D=X$, the only non-trivial equivalence class of $\sim$ is $\{x,y\}$ with $\coverof x=\coverof y$, and $x \in S$ iff $y \in S$.
\end{enumerate}
\end{theorem}

\begin{proof}
Clearly bounded nuclear subalgebras of $\upsets(X)$ correspond to nuclear partial correct partitions $\sim$ of $X$ with cofinal domain, and maximal bounded nuclear subalgebras are exactly the maximal nuclear subalgebras that are bounded. Therefore, the result follows from Theorem~\ref{thm:dual maximal nuclear subalg}.
\end{proof}

We next adjust the definition of irreducible and universal models to the setting of bounded nuclear implicative semilattices.

\begin{definition}
We say that a model $\mathfrak M=(X,S, \colorsof)$ is \textit{irreducible} for $\nisbot$ if the nuclear implicative semilattice $(\upsets(X),j_S)$ is generated by the upsets $U_1, \ldots, U_n$ as a bounded nuclear implicative semilattice.
\end{definition}

\begin{theorem}[Coloring Theorem for $\nisbot$] \label{thm:conditions for bounded irred model}
A finite model $\mathfrak M=(X, S,\colorsof)$ is irreducible for $\nisbot$  iff the following two conditions are satisfied:
\begin{enumerate}
\item $\forall x\in X \setminus \max X \left(\colorsof(x)=\colorsof(\coverof x)\Rightarrow x\in S\ \&\ {\coverof x}\nsubseteq S\right)$,
\item $\forall x,y\in X \left(\coverof x=\coverof y\ \&\ \colorsof(x)=\colorsof(y)\ \&\ (x\in S\Leftrightarrow y\in S)\Rightarrow x=y\right)$.
\end{enumerate}
\end{theorem}

\begin{proof}
The proof is analogous to the proof of Theorem~\ref{thm:conditions for irred model} but uses the dual characterization of maximal bounded nuclear subalgebras of Theorem~\ref{thm:dual maximal bounded nuclear subalg}.
\end{proof}

\begin{remark}
Recall that $\colorsof(\emptyset)=\{ 1, \ldots, n \}$. Thus, if $x \in \max X$, then $\colorsof (\coverof x)=\{ 1, \ldots, n \}$. While Theorem~\ref{thm:conditions for bounded irred model} does not exclude the existence of elements with color $\{ 1, \ldots, n \}$ in an irreducible model for $\nisbot$, the second condition of Theorem~\ref{thm:conditions for bounded irred model} implies that there can only be at most two such elements in $\max X$ and they cannot be both in $S$. Note that there can be elements with color $\{ 1, \ldots, n \}$ outside of $\max X$ (see Remark~\ref{rem:X^n in bounded universal model is finite}).
\end{remark}

\begin{definition}
Let $\mathfrak L=(X,S,\colorsof)$ be an irreducible model for $\nisbot$. We call $\mathfrak L$ \emph{$n$-universal} for $\nisbot$ provided for every finite irreducible model $\mathfrak M=(Y,T,\colorsof)$ for $\nisbot$ there is a unique embedding of posets $e:Y \to X$ such that $e(Y)$ is an upset of $X$, $e^{-1}(S)=T$, and $\colorsof (e(y))=\colorsof(y)$ for all $y \in Y$.
\end{definition}

The construction of the universal model $\mathfrak L=(X,S,\colorsof)$ for $\nisbot$ is similar to that of the universal model for $\nis$. The only difference is in the construction of the first layer $\mathfrak L_1$ where $\sigma$ is allowed to be the full color $\{ 1, \ldots, n \}$. The other layers are constructed as in Definition~\ref{def:universal model}.

\begin{remark} \label{rem:X^n in bounded universal model is finite}
The first layer of the universal model for $\nisbot$ has two elements with the full color $\{ 1, \ldots, n \}$, one of which is in $S$. In total there are exactly four elements with full color:
\begin{itemize}
\item $y_1=r_{\varnothing,\{ 1, \ldots, n \}}$,
\item $y_2=s_{\varnothing,\{ 1, \ldots, n \}}$,
\item $y_3=s_{\{y_1\},\{ 1, \ldots, n \}}$,
\item $y_4=s_{\{y_1,y_2\},\{ 1, \ldots, n \}}$.
\end{itemize}
Indeed, the elements with full color other than $y_1$ and $y_2$ can only be obtained by applying Rule (3) to the antichains $\{y_1\}$ and $\{y_1,y_2\}$, which yield $y_3$ and $y_4$. Note that $y_3,y_4 \in S$ and $y_3,y_4 \in \dn y_1$, so there is no new antichain to which Rule (3) applies. We have that $y_1,y_2 \in X_1$ and $y_3,y_4 \in X_2$. Thus, $\{ x \in X \mid \colorsof(x)=\{ 1, \ldots, n \} \}$ has height $2$.
\end{remark}

\begin{theorem}\label{thm:bounded universal model and it is finite}
The model $\mathfrak L=(X,S,\colorsof)$ constructed above is universal for $\nisbot$ and finite.
\end{theorem}

\begin{proof}
The proof of universality for $\nis$ proceeds as the proof of Theorem~\ref{thm:universal model}. The only difference is in the definition of the map $e$ on the elements of height $1$ since irreducible models for $\nisbot$ are allowed to have at most two maximal elements with full color. Those elements are sent either to $s_{\varnothing,\{ 1, \ldots, n \}}$ or $r_{\varnothing,\{ 1, \ldots, n \}}$ depending on whether or not they are in $S$.

We noted in Remark~\ref{rem:X^n in bounded universal model is finite} that
$\{ x \in X \mid \colorsof(x)=\{ 1, \ldots, n \} \}$
is finite. Therefore, the finiteness of $\mathfrak L$ can be proved as in Theorem~\ref{thm:universal model is finite}.
\end{proof}

In addition, a simple modification of the proof of Theorem~\ref{thm:nis generated by finite algebras}, where we let $B$ to be the bounded subalgebra of $(A,\wedge,\to,0)$ generated by $F$, yields the following:

\begin{theorem}\label{thm:nisbot generated by finite algebras}
$\nisbot$ is generated by its finite algebras.
\end{theorem}

Finally, by following the proof of Theorem~\ref{thm:nis locally finite} and using Theorems~\ref{thm:bounded universal model and it is finite} and~\ref{thm:nisbot generated by finite algebras}, we arrive at the main result of this section:

\begin{theorem}
$\nisbot$ is locally finite.
\end{theorem}

\section{An alternative proof of Diego's Theorem}\label{sec:alternative Diego's Theorem}

We cannot claim that Diego's Theorem is a consequence of Theorem~\ref{thm:nis locally finite} because we use it in the proof of Theorem~\ref{thm:nis generated by finite algebras} which is used in the proof of Theorem~\ref{thm:nis locally finite}. To give an alternative proof of Diego's Theorem using our technique, we need to prove that $\is$ is generated by its finite algebras.
In this section we show how to do this by utilizing the technique of the distributive envelope of a distributive semilattice \cite{BJ08,BJ11,BJ13}.
We also show that the Coloring Theorem and the construction of  universal models simplify dramatically for $\is$.
For yet another proof of Diego's Theorem, using the technique of filtrations, see \cite[Sec.~5.4]{CZ97}.

We can identify $\is$ with the full subcategory of $\nis$ given by the nuclear implicative semilattices in which the nucleus is the identity. Alternatively, we can identify $\is$ with the full subcategory of $\nis$ given by the nuclear implicative semilattices in which the nucleus maps every element to $1$. The finite algebras in these two subcategories of $\nis$ are dual to S-posets $(X,S)$ with $S=X$ or $S=\emptyset$. In either case, the subset $S$ is not giving any additional information and we can drop it from our consideration. Therefore, a model for $\is$ is simply a pair $\mathfrak M=(X, \colorsof)$ where $X$ is a poset and $\colorsof$ is a coloring, and $\mathfrak M$ is irreducible if $\upsets(X)$ is generated as an implicative semilattice by the upsets $U_1,\dots,U_n$ that the coloring gives rise to. Thus, the Coloring Theorem simplifies as follows.

\begin{theorem}[Coloring Theorem for $\is$]
A finite model $\mathfrak M=(X,\colorsof)$ is irreducible for $\is$ iff the following two conditions are satisfied:
\begin{enumerate}
\item $\forall x\in X \  \colorsof(x) \subset \colorsof(\coverof x)$,
\item $\forall x,y\in X \left(\coverof x=\coverof y\ \&\ \colorsof(x)=\colorsof(y) \Rightarrow x=y\right)$.
\end{enumerate}
\end{theorem}

The construction of the $n$-universal model $\mathfrak L$ for $\is$ also simplifies considerably since we only need to consider the elements $r_{\alpha, \sigma}$,
so only Rule (1) applies. Therefore, if $x <y$, then $\colorsof(x) \subset \colorsof(y)$. Thus,
the cardinality of the colors strictly decreases layer by layer. It is then clear that the height of $\mathfrak L$ is at most $n$. Thus, $\mathfrak L$ is finite because the construction ends at the $n$-th layer. In fact, the height of $\mathfrak L$ is exactly $n$ since we can construct a chain $x_n < \cdots < x_1$ where $x_1=r_{\emptyset, \{1, \ldots,n-1 \}}$, and $x_{k}=r_{\{x_{k-1}\}, \{1, \ldots,n-k \} }$ for all $k=2, \ldots, n$.

We next show that $\is$ is generated by its finite algebras. For this we require the notion of distributive envelope (see, e.g., \cite[Sec.~3]{BJ11}). Let $A$ be an implicative semilattice. We recall that a filter of $A$ is \textit{prime} if it is a meet-prime element in the lattice of all filters of $A$. Note that if $A$ is a lattice, then this notion coincides with the usual notion of a prime filter. Let $Y_A$ be the poset of prime filters of $A$ ordered by inclusion. It follows from the Prime Filter Lemma that the Stone map $\sigma:A \to \upsets(Y_A)$ given by $\sigma(a)=\{ y \mid a \in y \}$ is a meet-semilattice embedding.

\begin{definition}
The \textit{distributive envelope} $D(A)$ of $A$ is the sublattice of $\upsets(Y_A)$ generated by $\sigma(A)$.
\end{definition}

For various characterizations of $D(A)$ see~\cite[Sec.~3]{BJ11}.

\begin{theorem} \label{thm:is generated by finite algebras}
$\is$ is generated by its finite algebras.
\end{theorem}

\begin{proof}
Let $t(x_1, \ldots,x_n)$ be a term in the language of implicative semilattices such that the equation $t(x_1, \ldots, x_n) =1$ is not derivable from the equations defining $\is$. Then there is an implicative semilattice $A$ and $a_1, \ldots,a_n \in A$ such that $t(a_1, \ldots, a_n) \neq 1$ in $A$. Set
\[
F=\{ t'(a_1, \ldots,a_n) \mid \mbox{$t'$ is a subterm of $t$} \}.
\]
Since $A$ is embedded in $D(A)$, we identify $A$ with a subset of $D(A)$. Let $B$ be the sublattice of $D(A)$ generated by $F$. Because $D(A)$ is a distributive lattice, it is locally finite. Thus, $B$ is a finite distributive lattice, and hence an implicative semilattice. Therefore, for each $a, b \in B$ the relative pseudocomplent $a\to_B b$ exists in $B$. Suppose that the relative pseudocomplement $a\to_{D(A)}b$ exists in $D(A)$. We then have
\[
a \to_B b= \bigvee \{ x \in B \mid a \meet x \le b \} =\bigvee \{ x \in B \mid x \le a \to_{D(A)} b \}.
\]
Therefore, $a \to_B b \le a \to_{D(A)} b$. Moreover, if $a \to_{D(A)} b \in B$, then ${a \to_B b} = {a \to_{D(A)} b}$.
By \cite[Lem.~3.3]{BJ13}, if $a,b \in A$, then $a \to_{D(A)} b$ exists and it coincides with the relative pseudocomplement $a \to_A b$ in $A$. Thus, if $a,b \in A \cap B$ and $a \to_A b \in B$, then
$a \to_B b=a \to_{D(A)} b = a \to_A b$.
Therefore, for each subterm $t'$ of $t$, the computation of $t'(a_1, \ldots, a_n)$ in $A$ is the same as that in $B$. Thus, $t(a_1, \ldots, a_n) \neq 1$ in $A$ implies that $t(a_1, \ldots, a_n) \neq 1$ in $B$. Consequently, $t(x_1, \ldots, x_n) =1$ is refuted in the finite implicative semilattice $B$.
\end{proof}

Since $\is$ is generated by its finite algebras and $n$-universal models for $\is$ are finite, the same argument as in the proof of Theorem~\ref{thm:nis locally finite} yields Diego's Theorem:

\begin{theorem}
$\is$ is locally finite.
\end{theorem}

\begin{remark}
An analogous strategy can be employed to prove that the variety of bounded implicative semilattices is locally finite.
The characterization of irreducible models and the definition of the universal model have to be adjusted to allow elements with full color in the first layer. But the construction of the universal model terminates for the same reason as in the case of implicative semilattices.
\end{remark}

\section{Examples}\label{sec:examples}

In this final section we describe some $n$-universal models for $\nis$ and $\nisbot$ for small $n$. If $n=0$, then the only available color is the empty one. Therefore, the full and empty colors coincide, which yields that the $0$-universal model for $\nis$ is empty, and we arrive at the following theorem.

\begin{theorem}
The free $0$-generated nuclear implicative semilattice is the trivial algebra.
\end{theorem}

If $n=1$, then there are two colors: $\emptyset$ and $\{ 1 \}$, with $\{ 1 \}$ being the full color. Since there can be no element in the $1$-universal model with the full color, the first layer has two elements: $r_{\emptyset, \emptyset}$ and $s_{\emptyset, \emptyset}$. Rules (1) and (2) of Definition~\ref{def:universal model} allow us to add elements to the next layer only if their color is strictly smaller than the color of their cover. All the points in the first layer have empty color, so these rules do not apply. Rule (3) gives an element in $S$ with empty color for each antichain not contained in $S$. There are two such antichains: $\{r_{\emptyset, \emptyset}\}$ and $\{r_{\emptyset, \emptyset},s_{\emptyset, \emptyset}\}$. Therefore, the second layer of the $1$-universal model is made of the two elements $s_{\{r_{\emptyset, \emptyset}\}, \emptyset}$ and $s_{\{r_{\emptyset, \emptyset},s_{\emptyset, \emptyset}\}, \emptyset}$. The third layer is empty because Rules (1) and (2) do not apply since every element has empty color, and Rule (3) does not apply as every antichain that is not contained in $S$ is contained entirely in the first layer. Thus, the $1$-universal model has $4$ elements and its diagram is shown below.

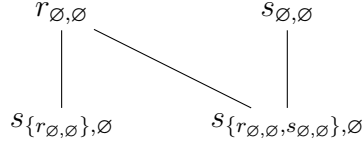
\begin{figure}[H]
\begin{tikzpicture}[inner sep = 2pt]
\node (r00) at (-3,0) {$r_{\emptyset, \emptyset}$};
\node (s00) at (0,0) {$s_{\emptyset, \emptyset}$};
\node (s10) at (-3,-1.5) {$s_{\{r_{\emptyset, \emptyset}\}, \emptyset}$};
\node (s11) at (0,-1.5) {$s_{\{r_{\emptyset, \emptyset},s_{\emptyset, \emptyset}\}, \emptyset}$};
\draw (s10) -- (r00);
\draw (s11) -- (r00);
\draw (s11) -- (s00);
\end{tikzpicture}
\caption{The $1$-universal model for $\nis$.}\label{fig:1-univ model for nis}
\end{figure}

The free $1$-generated nuclear implicative semilattice is isomorphic to the upsets of the $1$-universal model. Since $U_1=\{x \mid 1\in\colorsof(x)\}=\varnothing$, it is generated by the bottom element.
Recalling that
$j_S(U)=X \setminus \dn (S \setminus U)$
for each upset $U$, it is easy to see that we have
\[
\begin{array}{lll}
j_S(\emptyset) = \{r_{\emptyset, \emptyset}\} &
\qquad \neg j_S(\emptyset) = \{s_{\emptyset, \emptyset}\} &
\qquad \neg \neg j_S(\emptyset) = \{r_{\emptyset, \emptyset},s_{\{r_{\emptyset, \emptyset}\}, \emptyset}\} \\[2ex]
j_S \neg j_S(\emptyset) = \{r_{\emptyset, \emptyset},s_{\emptyset, \emptyset}\} &
\multicolumn{2}{l}{ \qquad \neg \neg j_S(\emptyset) \to j_S(\emptyset) = \{r_{\emptyset, \emptyset},s_{\emptyset, \emptyset}, s_{\{r_{\emptyset, \emptyset},s_{\emptyset, \emptyset}\}, \emptyset} \} } \\[2ex]
\multicolumn{3}{l}{ (\neg\neg j_S(\emptyset)\to j_S(\emptyset))\to j_S\neg j_S(\emptyset) = \{r_{\emptyset, \emptyset},s_{\emptyset, \emptyset},s_{\{r_{\emptyset, \emptyset}\}, \emptyset} \} }
\end{array}
\]
Therefore, letting
$g=\varnothing$
and abbreviating $a\to g$ by $\neg_g a$, we arrive at the following theorem.

\begin{theorem}
The free $1$-generated nuclear implicative semilattice is the algebra shown below.
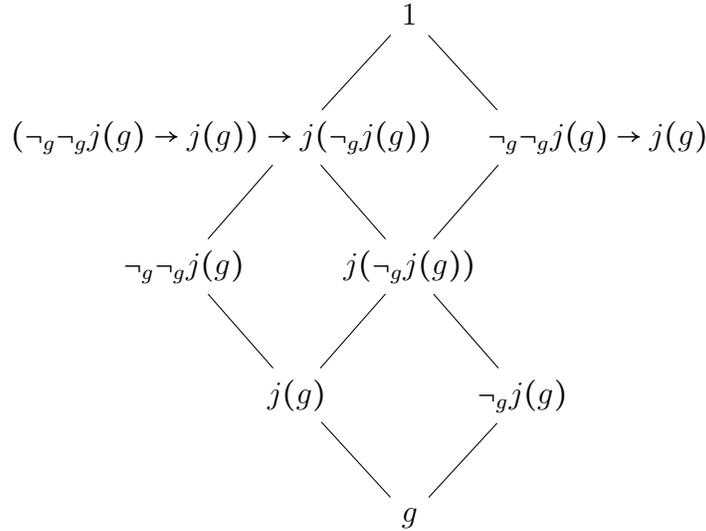
\begin{figure}[H]
\begin{tikzpicture}
\matrix[matrix of math nodes,column sep={1.5cm,between origins},row sep=1cm]
{
&&|(1)|1&\\
&|(r00s00s10)|\hspace{-2cm} ( \neg_g\neg_g j(g) \to j(g))\to j(\neg_g j(g)) &&|(r00s00s11)|\hspace{2cm} \neg_g\neg_g j(g) \to j(g) &\\
|(r00s10)| \neg_g\neg_g j(g) &&|(r00s00)| j(\neg_g j(g)) &\\
&|(r00)| j(g) &&|(s00)| \neg_g j(g) \\
&&|(empty)|g &\\
};
\draw (empty) -- (r00) -- (r00s00) -- (s00) -- (empty);
\draw (r00s00) -- (r00s00s10) -- (r00s10) -- (r00);
\draw (r00s00) -- (r00s00s11) -- (1) -- (r00s00s10);
\end{tikzpicture}
\caption{The free nuclear implicative semilattice on one generator $g$.}\label{fig:1-free alg for nis}
\end{figure}
\end{theorem}

The $2$-universal model for $\nis$ is already quite large. There are four colors: $\emptyset$, $\{1\}$, $\{2\}$, and $\{1,2\}$, with $\{1,2\}$ being the full color. For each non-full color, the first layer contains one element from $S$ and another outside of $S$. Thus, the first layer has $6$ elements. There are exactly $6$ antichains in the first layer that have nonempty color
(recall (\ref{color of a subset})).
For each such antichain $\alpha$, Rule (1) tells us to add the element $r_{\alpha, \emptyset}$ to the second layer. The first layer together with the elements of the second layer obtained by applying Rule (1) is shown below.
To make figures easier to follow, from now on elements of universal models will be denoted with only one subscript describing their color. Their cover will be clear from the figure.

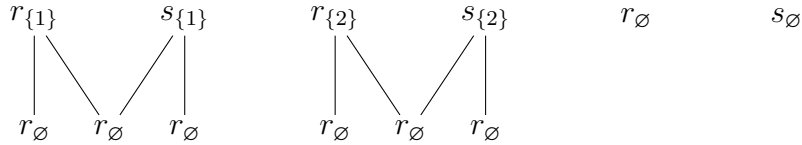
\begin{figure}[H]
\begin{tikzpicture}[inner sep = 2pt]
\node (r00) at (-6,0) {$r_{\{1\}}$};
\node (s00) at (-4,0) {$s_{\{1\}}$};
\node (r01) at (-2,0) {$r_{\{2\}}$};
\node (s01) at (0,0) {$s_{\{2\}}$};
\node (r02) at (2,0) {$r_\varnothing$};
\node (s02) at (4,0) {$s_\varnothing$};
\node (r10) at (-6,-1.5) {$r_\varnothing$};
\node (r11) at (-4,-1.5) {$r_\varnothing$};
\node (r12) at (-2,-1.5) {$r_\varnothing$};
\node (r13) at (0,-1.5) {$r_\varnothing$};
\node (r14) at (-5,-1.5) {$r_\varnothing$};
\node (r15) at (-1,-1.5) {$r_\varnothing$};
\draw (r10) -- (r00);
\draw (r11) -- (s00);
\draw (r12) -- (r01);
\draw (r13) -- (s01);
\draw (r00) -- (r14) -- (s00);
\draw (r01) -- (r15) -- (s01);
\end{tikzpicture}
\caption{The first layer and part of the second layer of the $2$-universal model for $\nis$.}\label{fig:2-univ model for nis}
\end{figure}

Similarly, Rule (2) forces us to add $6$ elements of the form $s_{\alpha, \emptyset}$ to the second layer. To apply Rule (3), we need to consider antichains contained in the first layer that are not entirely contained in $S$. There are many such antichains, each yielding an element of the form $s_{\alpha, \colorsof(\alpha)}$ added to the second layer. Thus, the second layer is rather large to draw easily. The construction of the $2$-universal model does not stop at the second layer because while Rules (1) and (2) do not apply anymore, Rule (3) still applies.
While we leave the details out, it can be estimated that the height of the $2$-universal model is 17.

From the above description of the $n$-universal models for $\nis$ where $n=0,1,2$, we can easily obtain a description of the corresponding universal models for $\is$. All we need to do is to take out the points that are in $S$ and the points that are obtained from  antichains containing some elements of $S$. In other words, we erase the downset of $S$ from the universal model for $\nis$.

Consequently, the $0$-universal model for $\is$ is empty. The $1$-universal model has one element $r_{\emptyset, \emptyset}$ and is obtained by erasing
$S$ in Figure~\ref{fig:1-univ model for nis}.
Thus, the free $0$-generated implicative semilattice is trivial, while the free $1$-generated implicative semilattice is a $2$-element chain, with the generator $g$ being the bottom element.

Erasing the downset of $S$ from the $2$-universal model for $\nis$ gives the $2$-universal model for $\is$ shown in Figure~\ref{fig:2-univ model for is}.

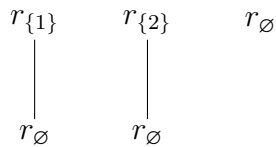
\begin{figure}[H]
\begin{tikzpicture}[inner sep = 2pt]
\node (r00) at (-1.5,0) {$r_{\{1\}}$};
\node (r01) at (0,0) {$r_{\{2\}}$};
\node (r02) at (1.5,0) {$r_\varnothing$};
\node (r10) at (-1.5,-1.5) {$r_\varnothing$};
\node (r11) at (0,-1.5) {$r_\varnothing$};
\draw (r10) -- (r00);
\draw (r11) -- (r01);
\end{tikzpicture}
\caption{The $2$-universal model for $\is$.}\label{fig:2-univ model for is}
\end{figure}

The $3$-universal model for $\is$ is a lot more complicated, and the $4$-universal model for $\is$ is practically impossible to describe (see~\cite[Sec.~4]{Koh81} for details).

We finish the paper by describing how things change in the bounded case. The key difference is that in the bounded case the full color becomes available. Thus, the $0$-generated universal model is no longer empty.
In fact, it has 4 elements. Indeed, since the full color $\emptyset$ is now allowed, the diagram of the $0$-universal model for $\nisbot$ is the same as the diagram of the $1$-universal model for $\nis$ given in Figure~\ref{fig:1-univ model for nis}. Therefore, we arrive at the following theorem.

\begin{theorem}
The free $0$-generated bounded nuclear implicative semilattice is isomorphic to the free $1$-generated nuclear implicative semilattice shown in Figure~\ref{fig:1-free alg for nis}.
\end{theorem}

Since we allow one more color in the first layer, the complexity of the $1$-universal model for $\nisbot$ is comparable with the complexity of the $2$-universal model for $\nis$. Therefore, the $1$-universal model for $\nisbot$ is already quite large. Despite its first layer having only $4$ elements $r_{ \emptyset,\{ 1 \}}$, $s_{\emptyset,\{ 1 \}}$, $r_{\emptyset, \emptyset}$, and $s_{\emptyset, \emptyset}$, the cardinality of next layers grows fast. Thus, it is not easy to draw it. Instead we describe fully the $1$-universal model for two subvarieties of $\nisbot$ where the nucleus is either dense or locally dense.
Such nuclei are used in the study of cofinal subframe superintuitionistic logics (see \cite{BG07}) and have applications in computer science (see \cite[Sec.~7]{FM97}).

\begin{definition}
A nucleus $j$ on a bounded implicative semilattice is \textit{dense} if $j(0)=0$, and it is \textit{locally dense} if $j(\neg j(0))=1$.
\end{definition}

Each dense nucleus is locally dense, but the converse is not true in general.
The following dual characterization of dense and locally dense nuclei on finite implicative semilattices follows from \cite{BG07}.

\begin{proposition}
Let $(X,S)$ be a finite S-poset.
\begin{enumerate}
\item $j_S$ is dense iff $\max X \subseteq S$.
\item $j_S$ is locally dense iff $\up S \cap \max X \subseteq S$.
\end{enumerate}
\end{proposition}

This implies that in universal models for the dense case only elements from $S$ are allowed in the first layer. On the other hand, in universal models for the locally dense case, when applying Rules (2) and (3), we add the element $s_{\alpha, \sigma}$ only when $X_1 \cap \up \alpha \subseteq S_1$.

In the dense case, the $0$-universal model has only one element $s_{\emptyset, \emptyset}$. Indeed, we only allow elements of $S$ in the first layer. Then the construction has to stop because there is no element with nonempty color, nor an element outside of $S$, and so we cannot apply any of the three rules. Thus, we arrive at the following theorem.

\begin{theorem}
The free $0$-generated dense nuclear implicative semilattice is a $2$-element chain on which $j$ is the identity.
\end{theorem}

In the locally dense case, the $0$-universal model has 2 incomparable elements $r_{\emptyset,\emptyset}, s_{\emptyset,\emptyset}$ in the first layer. Then the construction has to stop because Rules (1) and (2) do not apply since the only color is the empty color, and we are not allowed to apply Rule (3) because of the additional constraint in this case. Thus, we arrive at the following theorem.

\begin{theorem}
The free $0$-generated locally dense nuclear implicative semilattice
is the algebra shown below.
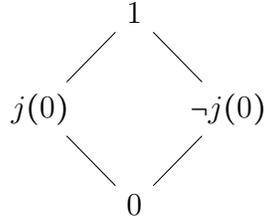
\begin{figure}[H]
\begin{tikzpicture}
\matrix[matrix of math nodes, column sep={1.25cm,between origins}, row sep = 0.65cm]
{
&|(1)|1 &\\
|(nj)|
j(0) &&|(j)|
\neg j(0) \\
&|(0)|
0 &\\
};
\draw (0) -- (j) -- (1) -- (nj) -- (0);
\end{tikzpicture}
\caption{The free $0$-generated locally dense nuclear implicative semilattice.}\label{fig:0-free alg for nisbotld}
\end{figure}
\end{theorem}

The $1$-universal model for the dense case has two elements in the first layer $s_{\emptyset, \{1\}}$ and $s_{\emptyset, \emptyset}$.
The only antichain to which we can apply Rules (1) and (2) is $\{ s_{\emptyset, \{1\}} \}$, yielding two elements in the second layer $r_{\{s_{\emptyset, \{1\}}\}, \emptyset}$ and $s_{\{s_{\emptyset, \{1\}}\}, \emptyset}$. Rule (3) does not apply in the construction of the second layer because all the elements of the first layer are in $S$. Since in the first two layers there is only one antichain with nonempty color $\{ s_{\emptyset, \{1\}} \}$ and it is contained in the first layer, we cannot apply Rules (1) and (2) anymore. However, we can apply Rule (3) to $4$ antichains $\{ r_{\{s_{\emptyset, \{1\}}\}, \emptyset} \}$, $\{ r_{\{s_{\emptyset, \{1\}}\}, \emptyset}, s_{\{s_{\emptyset, \{1\}}\}, \emptyset} \}$, $\{ r_{\{s_{\emptyset, \{1\}}\}, \emptyset}, s_{\emptyset,\emptyset} \}$, and $\{ r_{\{s_{\emptyset, \{1\}}\}, \emptyset}, s_{\{s_{\emptyset, \{1\}}\}, \emptyset}, s_{\emptyset,\emptyset} \}$. This gives $4$ elements in the third layer that are all in $S$. Then the construction has to stop because we cannot apply Rules (1) and (2) since all the new elements added have empty color, and Rule (3) does not apply because we cannot find any antichain containing $r_{\{s_{\emptyset, \{1\}}\}, \emptyset}$ that is not entirely contained in the first two layers. Thus, the $1$-universal model for the dense case looks as follows,
where we recall from the above that we only use one subscript describing the  color of an element.

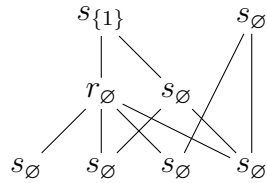
\begin{figure}[H]
\begin{tikzpicture}[inner sep = 2pt]
\node (s00) at (0,0) {$s_\varnothing$};
\node (s01) at (-2,0) {$s_{\{1\}}$};
\node (r10) at (-2,-1) {$r_\varnothing$};
\node (s10) at (-1,-1) {$s_\varnothing$};
\node (s20) at (0,-2) {$s_\varnothing$};
\node (s21) at (-1,-2) {$s_\varnothing$};
\node (s22) at (-2,-2) {$s_\varnothing$};
\node (s23) at (-3,-2) {$s_\varnothing$};
\draw (r10) -- (s01);
\draw (s10) -- (s01);
\draw (s00) -- (s20) -- (r10);
\draw (s00) -- (s21) -- (r10);
\draw (s20) -- (s10);
\draw (s10) -- (s22) -- (r10);
\draw (s23) -- (r10);
\end{tikzpicture}
\caption{The $1$-universal model for the dense case.}\label{fig:1-univ model for nisbotd}
\end{figure}

In the $1$-universal model for the locally dense case we have two additional elements $r_{\emptyset,\{ 1 \}}$ and $r_{\emptyset, \emptyset}$ in the first layer, and they contribute to two additional elements $r_{\{r_{\emptyset,\{ 1 \}} \}, \emptyset}$, $r_{\{r_{\emptyset,\{ 1 \}}, s_{\emptyset,\{ 1 \}} \}, \emptyset}$ in the second layer thanks to Rule (1). But since we are not allowed to add elements in $S$ unless the maximal elements of their upsets are all in $S$, these new elements do not contribute to adding any new element in the third layer. Thus, the construction stops.

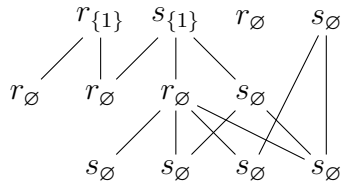
\begin{figure}[H]
\begin{tikzpicture}[inner sep = 2pt]
\node (s00) at (0,0) {$s_\varnothing$};
\node (s01) at (-2,0) {$s_{\{1\}}$};
\node (r00) at (-1,0) {$r_\varnothing$};
\node (r01) at (-3,0) {$r_{\{1\}}$};
\node (r10) at (-2,-1) {$r_\varnothing$};
\node (r11) at (-3,-1) {$r_\varnothing$};
\node (r12) at (-4,-1) {$r_\varnothing$};
\node (s10) at (-1,-1) {$s_\varnothing$};
\node (s20) at (0,-2) {$s_\varnothing$};
\node (s21) at (-1,-2) {$s_\varnothing$};
\node (s22) at (-2,-2) {$s_\varnothing$};
\node (s23) at (-3,-2) {$s_\varnothing$};
\draw (r10) -- (s01);
\draw (s10) -- (s01);
\draw (r01) -- (r11) -- (s01);
\draw (r12) -- (r01);
\draw (s00) -- (s20) -- (r10);
\draw (s00) -- (s21) -- (r10);
\draw (s20) -- (s10);
\draw (s10) -- (s22) -- (r10);
\draw (s23) -- (r10);
\end{tikzpicture}
\caption{The $1$-universal model for the locally dense case.}\label{fig:1-univ model for nisbotld}
\end{figure}

\bibliographystyle{amsplain}
\bibliography{bibliography}

\end{document}